\documentclass[11pt, a4paper, reqno, english]{smfart}

\usepackage{hyperref}
\usepackage[all]{xy}
\usepackage{amssymb}
\usepackage{enumerate}

\newcommand\PP{\mathbb{P}}
\newcommand\Ptwo{{\PP^2}}
\newcommand\Pfour{{\PP^4}}
\newcommand\ZZ{\mathbb{Z}}
\newcommand\ZZp{\ZZ_{>0}}
\newcommand\ZZnn{\ZZ_{\ge 0}}
\newcommand\ZZnz{\ZZ_{\ne 0}}
\newcommand\QQ{\mathbb{Q}}
\newcommand\RR{\mathbb{R}}
\newcommand\RRp{\RR_{>0}}
\newcommand\RRnn{\RR_{\ge 0}}
\newcommand\RRnz{\RR_{\ne 0}}
\newcommand\Ga{\mathbb{G}_\mathrm{a}}

\newcommand\A{\mathcal{A}}
\newcommand\E{\mathcal{E}}
\newcommand\T{\mathcal{T}}
\newcommand\R{\mathcal{R}}
\renewcommand\P{\mathcal{P}}
\newcommand\M{\mathcal{M}}

\newcommand{\Tu}[1]{\T_{#1}}
\newcommand{\TS}{\Tu S}
\newcommand{\TtS}{\Tu \tS}
\newcommand{\bA}{\mathbf A}
\newcommand{\Aone}{{\mathbf A}_1}
\newcommand{\Atwo}{{\mathbf A}_2}
\newcommand{\Athree}{{\mathbf A}_3}
\newcommand{\Afour}{{\mathbf A}_4}
\newcommand{\Afive}{{\mathbf A}_5}
\newcommand{\Dfour}{{\mathbf D}_4}
\newcommand{\Dfive}{{\mathbf D}_5}
\newcommand{\Esix}{{\mathbf E}_6}

\newcommand{\tS}{{\widetilde S}}
\newcommand\inj{\hookrightarrow}
\newcommand\xx{\mathbf{x}}
\renewcommand\tt{\mathbf{t}}
\newcommand\dd{\ \mathrm{d}}
\newcommand{\congr}[3]{{#1} \equiv {#2}\ (\mathrm{mod}\ {#3})}
\newcommand{\base}[7]{\ee^{({#1},{#2},{#3},{#4},{#5},{#6},{#7})}}
\DeclareMathOperator{\hcf}{gcd}
\DeclareMathOperator{\rk}{rk}
\DeclareMathOperator{\Pic}{Pic}
\DeclareMathOperator{\Cox}{Cox}
\DeclareMathOperator{\Spec}{Spec}
\DeclareMathOperator{\vol}{Vol}
\DeclareMathOperator\supp{supp}

\newcommand\al{\alpha}
\newcommand\e{\eta}
\newcommand{\ee}{\boldsymbol{\eta}}
\renewcommand{\aa}{\boldsymbol{\alpha}}
\newcommand{\bb}{\boldsymbol{\beta}}
\newcommand{\cc}{\boldsymbol{\gamma}}
\newcommand{\cp}[2]{{\hcf(#1,#2)=1}}
\newcommand{\ncp}[2]{{\hcf(#1,#2) > 1}}
\newcommand\phis{\phi^*}
\newcommand\phid{\phi^\dagger}
\newcommand\kk{\mathbf{k}}
\renewcommand\div{|}
\newcommand\ndiv{\nmid}
\newcommand\ddiv{||}
\newcommand\where{\mid}
\newcommand\Where{\,\Big|\,}
\newcommand\WHERE{\,\Bigg|\,}
\renewcommand\theta{\vartheta}
\renewcommand\rho{\varrho}
\newcommand\rto{\dashrightarrow}

\newcommand\sums[1]{\sum_{\substack{#1}}}
\newcommand\maxs[1]{\max_{\substack{#1}}}
\newcommand\mins[1]{\min_{\substack{#1}}}
\newcommand\prods[2]{\prod_{\substack{#1}}#2}

\newcommand\prodp[1]{\prod_p #1}
\newcommand\prodpp[1]{\prod_p\left(#1\right)}
\newcommand\prodpd[2]{\prod_{p\div #1}#2}
\newcommand\prodpdd[3]{\prod_{p^{#1}\ddiv #2}#3}
\newcommand\prodpn[2]{\prod_{p\ndiv #1}#2}
\newcommand\prodpdp[2]{\prodpd{#1}{\left(#2\right)}}
\newcommand\prodpddp[3]{\prodpdd{#1}{#2}{\left(#3\right)}}
\newcommand\prodpnp[2]{\prodpn{#1}{\left(#2\right)}}
\newcommand\prodpdn[3]{\prods{p\div #1\\p\ndiv #2}{#3}}
\newcommand\prodpddd[4]{\prods{p^{#1}\ddiv #2\\p\div #3}{#4}}
\newcommand\prodpddn[4]{\prods{p^{#1}\ddiv #2\\p\ndiv #3}{#4}}
\newcommand\prodpdnp[3]{\prodpdn{#1}{#2}{\left(#3\right)}}

\newcommand\Thetabox[2]{{\left[\begin{gathered}#1\\\text{Definition~#2}\end{gathered}\right]}}
\newcommand\injref[2]{\ar@{->}[#1]_-{\text{#2}}}
\newcommand\injreff[2]{\ar@{->}[#1]^-{\text{#2}}}
\newcommand\injrefff[2]{\ar@{->}[#1]_-{\text{#2}}}

\newtheorem{theorem}{Theorem}[section]
\newtheorem*{theorem*}{Theorem}
\newtheorem{lemma}[theorem]{Lemma}
\newtheorem{prop}[theorem]{Proposition}
\newtheorem{cor}[theorem]{Corollary}
\newtheorem*{question*}{Question}
\newtheorem*{conj*}{Conjecture}
\theoremstyle{definition}
\newtheorem{definition}[theorem]{Definition}
\theoremstyle{remark}
\newtheorem{remark}[theorem]{Remark}
\newtheorem{example}[theorem]{Example}

\numberwithin{equation}{section}
\numberwithin{table}{section}
\numberwithin{figure}{section}

\begin{document}

\title{Counting integral points on universal torsors}

\author{Ulrich Derenthal}

\address{Institut f\"ur Mathematik, Universit\"at Z\"urich,
  Winterthurerstrasse 190, 8057 Z\"urich, Switzerland}

\email{ulrich.derenthal@math.unizh.ch}

\begin{abstract}
  Manin's conjecture for the asymptotic behavior of the number of rational
  points of bounded height on del Pezzo surfaces can be approached through
  universal torsors. We prove several auxiliary results for the estimation of
  the number of integral points in certain regions on universal torsors.  As
  an application, we prove Manin's conjecture for a singular quartic del Pezzo
  surface.
\end{abstract}

\subjclass{11G35 (11D45, 14G05)}

\maketitle

\tableofcontents

\section{Introduction}\label{sec:counting_introduction}

The distribution of rational points on smooth and singular del Pezzo surfaces
is predicted by a conjecture of Yu.~I.~Manin \cite{MR89m:11060}.  For a del
Pezzo surface $S$ of degree $d \ge 3$ defined over the field $\QQ$ of rational
numbers, we consider a height function $H$ induced by an anticanonical
embedding of $S$ into $\PP^d$, where $H(\xx) = \max\{|x_0|, \dots, |x_d|\}$
for $\xx \in S(\QQ) \subset \PP^d(\QQ)$ represented by coprime integral
coordinates $x_0, \dots, x_d$.
Manin's conjecture makes the following prediction for the asymptotic behavior
of the number of rational points of height at most $B$ on the complement $U$
of the lines on $S$. As $B \to \infty$,
\begin{equation*}
  N_{U,H}(B) = \#\{\xx \in U(\QQ) \where H(\xx) \le B\} \sim cB(\log B)^{k-1},
\end{equation*}
where $k$ is the rank of the Picard group of $S$ (resp.\ of its minimal
desingularization if $S$ is a singular del Pezzo surface) and the leading
constant $c$ has a conjectural interpretation due to E.~Peyre
\cite{MR1340296}.

One approach to Manin's conjecture for del Pezzo surfaces uses \emph{universal
  torsors}. This approach was introduced by P.~Salberger \cite{MR1679841} in
the case of toric varieties. It also lead to the proof of Manin's conjecture
for some non-toric del Pezzo surfaces that are \emph{split}, i.e., all of
whose lines are defined over $\QQ$: quartic del Pezzo surfaces with a
singularity of type $\Dfive$ \cite{MR2320172}, $\Dfour$ \cite{MR2290499}
resp.\ $\Afour$ \cite{arXiv:0710.1560}, and a cubic surface with $\Esix$
singularity \cite{MR2332351}.
  
These proofs of Manin's conjecture for a split del Pezzo surface $S$ consist
of three main steps.
\begin{enumerate}[(1)]
\item\label{it:strategy_passage} One constructs an explicit bijection
  between rational points of bounded height on $S$ and integral points
  in a region on a universal torsor $\TS$.
\item\label{it:strategy_counting} Using methods of analytic number theory,
  one estimates the number of integral points in this region on the torsor
  by its volume.
\item\label{it:strategy_manipulation} One shows that the volume of this region
  grows asymptotically as predicted by Yu.~I.~Manin and E.~Peyre.
\end{enumerate}

Step~\ref{it:strategy_passage} is the focus of joint work with Yu.~Tschinkel
\cite[Section~4]{MR2290499}, giving a geometrically motivated approach to
determine a parameterization of the rational points on $S$ by integral points
on a universal torsor explicitly.

For step~\ref{it:strategy_counting}, we estimate the number of integral points
on the $(k+2)$-dimensional variety $\TS$ by performing $k+2$ summations over
one \emph{torsor variable} after the other; the remaining torsor variables are
determined by the \emph{torsor equations} defining $\TS$ as an affine variety.
In each summation, the main problem is to show that an error term summed over
the remaining variables gives a negligible contribution; see
Section~\ref{sec:first_summation} for the error term of the first summation in
a certain setting.

For these summations, the previous articles rely on some auxiliary analytic
results dealing with the average order of certain arithmetic functions over
intervals that are proved in a specific setting.  In this article, we
harmonize and generalize many of the analytic tools that have been brought to
bear so far; see Figure~\ref{fig:functions} for an overview of the sets of
arithmetic functions that we introduce. We expect that our results can be
applied to many different del Pezzo surfaces, at least to cover the more
standard bits of the argument. This will allow future work on Manin's
conjecture for del Pezzo surfaces to concentrate on the essential difficulties
in the estimation of some of the error terms, without having to reimplement
the routine parts.

As an application of our general techniques, we prove Manin's conjecture in a
new case: a quartic del Pezzo surface with singularity type $\Athree+\Aone$
(Section~\ref{sec:qa1a3}). This example also demonstrates how we can deal with
a new geometric feature. In the final $k$ summations, the previous proofs of
Manin's conjecture for split del Pezzo surfaces made crucial use of the fact
that the nef cone (the dual of the effective cone with respect to the
intersection form) is simplicial (in the quartic $\Dfive$ and $\Dfour$ cases
and in the cubic $\Esix$ case) or at least the difference of two simplicial
cones (in the quartic $\Afour$ case). The nef cone of the quartic surface
treated here has neither of these shapes. However, the techniques introduced
in Section~\ref{sec:generic_completion} are not sensitive to the shape of the
nef cone. In our example, they allow to handle the final $k+1=7$ summations
at the same time.

In fact, we expect that the techniques of Section~\ref{sec:generic_completion}
will cover the final $k$ summations for any del Pezzo surface. This would
narrow done the main difficulty of the universal torsor strategy to the
estimation of the error term in the first and second summation of
step~\ref{it:strategy_counting}.  For example, in recent joint work with
T.~D.~Browning, a proof of Manin's conjecture for a cubic surface with
$\Dfive$ singularity \cite{arXiv:0807.4733}, we make extensive use of the
results in this article to handle the final seven of nine summations, so that
we can focus on the considerable additional technical effort that is needed to
estimate the first two error terms.

Step~\ref{it:strategy_manipulation} is mixed with the second step in the basic
examples of the quartic $\Dfive$ \cite{MR2320172}, $\Dfour$ \cite{MR2290499}
and cubic $\Esix$ \cite{MR2332351} surfaces. However, it seems more natural to
treat the third step separately in more complicated cases, motivated by the
shape of the polytope whose volume appears in the leading constant. First
examples of this can be found in the treatment of the quartic $\Afour$
\cite{arXiv:0710.1560} and cubic $\Dfive$ \cite{arXiv:0807.4733} surfaces, and
we take the same approach in our example in Section~\ref{sec:qa1a3}.

\medskip\noindent\textbf{Acknowledgment.} The author thanks T.~D.~Browning and
the referee for their comments leading to improvements in the exposition of
this paper. He was partially supported by a Feodor Lynen Research Fellowship
of the Alexander von Humboldt Foundation and DFG grant DE 1646/1-1.

\section{The first summation}\label{sec:first_summation}

Let $S \subset \PP^d$ be an anticanonically embedded singular del Pezzo
surface of degree $d \ge 3$, with minimal desingularization $\tS$.
The first step of the universal torsor approach is to translate the counting
problem from rational points on $S$ to integral points on a universal torsor
$\TtS$. Then the number $N_{U,H}(B)$ of rational points of height at most $B$
on the complement $U$ of the lines on $S$ is the number of integral solutions
to the equations defining $\TtS$ that satisfy certain explicit coprimality
conditions and height conditions.

In several cases (see Remark~\ref{rem:special_form}), the counting problem on
$\TtS$ has the following \emph{special form}: $N_{U,H}(B)$ equals the
number of $(\al_0,\beta_0, \gamma_0, \aa, \bb, \cc, \delta)$ satisfying
\begin{itemize}
\item $(\al_0, \beta_0, \gamma_0) \in \ZZ_* \times \ZZ \times \ZZ$, where
  $\ZZ_*$ is $\ZZ$ or $\ZZnz$, $\aa=(\al_1,\dots,\al_r) \in \ZZp^{r}$,
  $\bb=(\beta_1,\dots,\beta_s) \in \ZZp^s$, $\cc=(\gamma_1, \dots, \gamma_t)
  \in \ZZp^t$, $\delta \in \ZZp$.
\item one torsor equation of the form
  \begin{equation}\label{eq:generic_torsor}
    \al_0^{a_0}\al_1^{a_1}\cdots \al_r^{a_r}+ \beta_0^{b_0}\beta_1^{b_1}\cdots
    \beta_s^{b_s}+\gamma_0\gamma_1^{c_1}\cdots\gamma_t^{c_t}=0,
  \end{equation}
  with $(a_0, \dots, a_r) \in \ZZp^{r+1}$, $(b_0, \dots, b_s) \in \ZZp^{s+1}$,
  $(c_1, \dots, c_t) \in \ZZp^t$. In particular, $\gamma_0$ appears linearly
  in the torsor equation, while $\delta$ does not appear.
\item height conditions that are written independently of $\gamma_0$
  (which can be achieved using~(\ref{eq:generic_torsor})) as
  \begin{equation}\label{eq:generic_height}
    h(\al_0, \beta_0, \aa, \bb, \cc, \delta; B) \le 1,
  \end{equation} for some function $h: \RR^{r+s+t+3} \times \RR_{\ge 3} \to
  \RR$. We assume that $h(\al_0,\beta_0,\aa,\bb,\cc,\delta;B) \le 1$ if and
  only if $\beta_0$ is in a union of finitely many intervals $I_1, \dots, I_n$
  whose number
  $n = n(\al_0,\aa,\bb,\cc,\delta;B)$ is bounded independently of
  $\al_0,\aa,\bb,\cc,\delta$ and $B$. By adding some empty intervals if
  necessary, we may assume that $n$ does not depend on
  $\al_0,\aa,\bb,\cc,\delta$ and $B$. For $j=1, \dots, n$, let $t_{0,j},
  t_{1,j}$ be the start and end point of $I_j$.
\item coprimality conditions that are described by
  Figure~\ref{fig:generic_dynkin} in the following sense. Let $A_i$
  (resp.\ $B_i$, $C_i$, $D$) correspond to $\al_i$ (resp. $\beta_i$,
  $\gamma_i$, $\delta$). Then two coordinates are required to be coprime if
  and only if the corresponding vertices in Figure~\ref{fig:generic_dynkin}
  are not connected by an edge. For variables corresponding to triples of
  pairwise connected symbols (besides $A_0,B_0,C_0$, this happens for triples
  consisting of $D$ and two of $A_0,B_0,C_0$ if at least two of $r,s,t$
  vanish), we assume that $\al_0, \beta_0, \gamma_0$ are allowed to have any
  common factor, while each prime dividing $\delta$ may divide at most one of
  $\al_0, \beta_0, \gamma_0$.
\end{itemize}

\begin{figure}[ht]
  \centering
  \[\xymatrix{A_0 \ar@{-}[r] \ar@{-}[dr] \ar@{-}[dd] & A_r  \ar@{-}[r] &
    A_{r-1} \ar@{-}[r] & \dots \ar@{-}[r] & A_1 \ar@{-}[dr]\\
    & B_0 \ar@{-}[r] & B_s \ar@{-}[r] & \dots \ar@{-}[r] & B_1 \ar@{-}[r] & D\\
    C_0 \ar@{-}[ur] \ar@{-}[r] & C_t \ar@{-}[r] & C_{t-1} \ar@{-}[r] & \dots
    \ar@{-}[r] & C_1 \ar@{-}[ur]}\]
  \caption{Extended Dynkin diagram}
  \label{fig:generic_dynkin}
\end{figure}
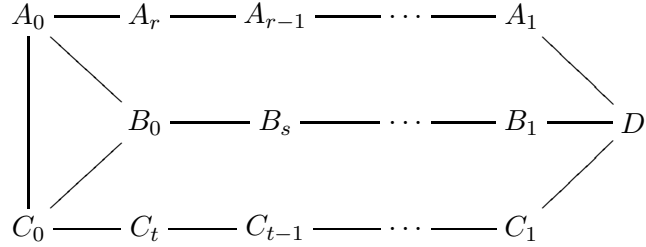

\begin{remark}\label{rem:special_form}
  The geometric background of this special form is as follows. A natural
  realization of a universal torsor $\TtS$ as an open subset of an affine
  variety is provided by \[\TtS \inj \Spec(\Cox(\tS))\]
  \cite[Theorem~5.6]{hassett_clay}. The coordinates of the affine variety
  $\Spec(\Cox(\tS))$ correspond to generators of the Cox ring of $\tS$.

  In \cite{math.AG/0604194}, we have classified singular del Pezzo surfaces
  $S$ of degree $d \ge 3$ where $\Spec(\Cox(\tS))$ is defined by precisely one
  \emph{torsor equation}. It includes the \emph{extended Dynkin diagrams}
  describing the configuration of the divisors on $\tS$ that correspond to the
  generators of $\Cox(\tS)$. In many cases, the extended Dynkin diagram has
  the special shape of Figure~\ref{fig:generic_dynkin}; see
  Table~\ref{tab:special_form} for their singularity types. In all cases
  besides one of the two isomorphy classes of cubic surfaces of type $\Dfour$,
  the torsor equation has the form of equation~\eqref{eq:generic_torsor}.

    \begin{table}[ht]
    \begin{center}
      \begin{equation*}
        \begin{array}{|c||c|c|}
          \hline
          \text{degree} &\text{shape of Figure~\ref{fig:generic_dynkin}} &
          \text{different shape}\\
          \hline\hline
          6 & \Aone,\ \Atwo & -\\
          5 & \Atwo,\ \Athree,\ \Afour & \Aone\\
          4 & \Athree,\ \Athree+\Aone,\ \Afour,\ \Dfour,\ \Dfive &
          3\Aone,\ \Atwo+\Aone\\
          3 & \Afour+\Aone,\ \Afive+\Aone,\ \Dfour,\ \Dfive,\ \Esix
          & \Athree+2\Aone,\ 2\Atwo+\Aone\\
          \hline
        \end{array}
      \end{equation*}
    \end{center}
    \caption{Extended Dynkin diagrams in \cite{math.AG/0604194}.}
    \label{tab:special_form}
  \end{table}

  If we construct the bijection between rational points on $S$ and integral
  points on $\TtS$ using the geometrically motivated approach of
  \cite[Section~4]{MR2290499}, then we expect to obtain coprimality conditions
  that are encoded in the extended Dynkin diagram.

  Indeed, in the quartic $\Dfour$ \cite{MR2290499}, $\Afour$
  \cite{arXiv:0710.1560} and the cubic $\Dfive$ \cite{arXiv:0807.4733} cases,
  both the extended Dynkin diagram and the counting problem have the special
  form. In the quartic $\Dfive$ \cite{MR2320172} and cubic $\Esix$
  \cite{MR2332351} cases, the extended Dynkin diagram has the shape of
  Figure~\ref{fig:generic_dynkin}, but the coprimality conditions are
  different. The reason is that the bijection between rational points on the
  del Pezzo surface and integral points on a universal torsor is constructed
  by ad-hoc manipulations of the defining equations. If one uses the method of
  \cite[Section~4]{MR2290499} instead, the coprimality conditions turn out in
  the expected shape.
\end{remark}

Given a counting problem of the special form above, we show in the remainder
of this section how to perform a first step towards estimating
$N_{U,H}(B)$. This will result in Proposition~\ref{prop:generic_first_step}.

Our first step can be described as follows, ignoring the coprimality
conditions for the moment. We determine the number of $\beta_0,\gamma_0$
satisfying the torsor equation (\ref{eq:generic_torsor}) while the other
coordinates are fixed. For any $\beta_0$
satisfying \[\congr{\al_0^{a_0}\al_1^{a_1}\cdots \al_r^{a_r}}
{-\beta_0^{b_0}\beta_1^{b_1}\cdots \beta_s^{b_s}}
{\gamma_1^{c_1}\cdots\gamma_t^{c_t}},\] there is a unique $\gamma_0$ such that
(\ref{eq:generic_torsor}) holds. Our assumption that the height conditions are
written as $h(\al_0,\beta_0,\aa,\bb,\cc,\delta;B) \le 1$ (independently of
$\gamma_0$) has the advantage that the number of $\beta_0,\gamma_0$ subject to
(\ref{eq:generic_torsor}) and (\ref{eq:generic_height}) is the number of
integral $\beta_0$ that lie in a certain subset $I$ of the real numbers
described by this height condition and satisfy the congruence above. If
$b_0=1$, one expects that this number is the measure of $I$ divided by the
modulus $\gamma_1^{c_1}\cdots\gamma_t^{c_t}$, with an error of $O(1)$.

Before coming to the details of this argument, we reformulate the coprimality
conditions.

\begin{definition}
  Let
  \begin{align*}
    \Pi(\aa) &= \al_1^{a_1}\cdots \al_r^{a_r}, &
    \Pi'(\delta,\aa)&=
    \begin{cases}
      \delta\al_1\cdots \al_{r-1}, &r \ge 1,\\
      1, &r=0,
    \end{cases}
  \end{align*}
  and we define $\Pi(\bb),\Pi'(\delta,\bb),\Pi(\cc),\Pi'(\delta,\cc)$
  analogously.
\end{definition}

\begin{lemma}
  Assume that $(\al_0, \beta_0, \gamma_0, \aa, \bb, \cc, \delta) \in
  \ZZ^{r+s+t+4}$ satisfies the torsor equation \eqref{eq:generic_torsor}.

  The coprimality conditions described by Figure~\ref{fig:generic_dynkin} hold
  if and only if
  \begin{align}
    \label{eq:gen_cpa} &\cp{\al_0}{\Pi'(\delta,\aa)\Pi(\bb)\Pi(\cc)},\\
    \label{eq:gen_cpb} &\cp{\beta_0}{\Pi'(\delta,\bb)\Pi(\aa)},\\
    \label{eq:gen_cpc} &\cp{\gamma_0}{\Pi'(\delta,\cc)},\\
    \label{eq:gen_cpabc} &\text{coprimality conditions for
      $\aa,\bb,\cc,\delta$ as in Figure~\ref{fig:generic_dynkin} hold.}
  \end{align}
\end{lemma}

\begin{proof}
  We must show that conditions (\ref{eq:gen_cpa})--(\ref{eq:gen_cpabc})
  together with (\ref{eq:generic_torsor}) imply
  $\cp{\beta_0}{\Pi(\cc)}$ and
  $\cp{\gamma_0}{\Pi(\aa)\Pi(\bb)}$.

  Suppose a prime $p$ divides $\gamma_0,\Pi(\aa)$, i.e., $p$ divides the first
  and third term of (\ref{eq:generic_torsor}). Then $p$ also divides the
  second term, $\beta_0^{b_0}\Pi(\bb)$. However, by (\ref{eq:gen_cpb}) and
  (\ref{eq:gen_cpabc}), we have $\cp{\beta_0^{b_0}\Pi(\bb)}{\Pi(\aa)}$. The
  remaining statements are proved analogously.
\end{proof}

For fixed $B \in \RR_{\ge 3}$ and $(\al_0,\aa,\bb,\cc,\delta) \in \ZZ_*
\times \ZZp^{r+s+t+1}$ subject to (\ref{eq:gen_cpa}), (\ref{eq:gen_cpabc}),
let $N_1=N_1(\al_0, \aa, \bb, \cc, \delta; B)$ be the number of
$\beta_0,\gamma_0$ subject to the torsor equation (\ref{eq:generic_torsor}),
the coprimality conditions (\ref{eq:gen_cpb}), (\ref{eq:gen_cpc}) and the
height condition $h(\al_0,\beta_0,\aa,\bb,\cc,\delta;B) \le 1$. Then
\[N_{U,H}(B) = \sums{(\al_0,\aa,\bb,\cc,\delta) \in \ZZ_* \times
  \ZZp^{r+s+t+1}\\\text{(\ref{eq:gen_cpa}), (\ref{eq:gen_cpabc}) hold}}
N_1(\al_0, \aa, \bb, \cc, \delta; B).\] Our goal is to find an
estimation for $N_1$, with an error term whose sum over $\al_0, \aa, \bb, \cc,
\delta$ is small.

First, we remove (\ref{eq:gen_cpc}) by a M\"obius inversion to obtain that
\[N_1 = \sum_{k_c\div \Pi'(\delta,\cc)} \mu(k_c)
\#\left\{\beta_0,\gamma_0' \in \ZZ \Where
  \begin{aligned}
    &\al_0^{a_0}\Pi(\aa) + \beta_0^{b_0}\Pi(\bb) +
    k_c\gamma_0'\Pi(\cc) = 0,\\
    &\text{(\ref{eq:gen_cpb}), $h(\al_0,\beta_0,\aa,\bb,\cc,\delta;B)
      \le 1$}
  \end{aligned}
\right\}.\] The torsor equation determines $\gamma_0'$ uniquely if a
congruence is fulfilled, so
\[N_1 = \sum_{k_c\div \Pi'(\delta,\cc)} \mu(k_c)
\#\left\{\beta_0 \in \ZZ \Where
  \begin{aligned}
    &\congr{\al_0^{a_0}\Pi(\aa)}{-\beta_0^{b_0}\Pi(\bb)}
    {k_c\Pi(\cc)},\\
    &\text{(\ref{eq:gen_cpb}), $h(\al_0,\beta_0,\aa,\bb,\cc,\delta;B) \le 1$}
  \end{aligned}
\right\}.\] This congruence cannot be fulfilled unless
$\cp{k_c}{\al_0\Pi(\aa)\Pi(\bb)}$. Indeed, if a prime $p$ divides $k_c$ and
$\al_0^{a_0}\Pi(\aa)$, then it divides also $\beta_0^{b_0}\Pi(\bb)$, but
$\cp{\Pi(\aa)}{\beta_0^{b_0}\Pi(\bb)}$ by (\ref{eq:gen_cpb}) and
(\ref{eq:gen_cpabc}), while $\cp{\al_0}{\Pi(\bb)}$ by (\ref{eq:gen_cpa}), and
$p \div k_c, \al_0, \beta_0$ is impossible because of (\ref{eq:gen_cpa}) and
since $p \div \delta, \al_0, \beta_0$ is not allowed by assumption; $p$
dividing $k_c$ and $\Pi(\bb)$ can be excluded similarly.  Therefore, we may
add the restriction $\cp{k_c}{\al_0\Pi(\aa)\Pi(\bb)}$ to the summation over
$k_c$ without changing the result, so that \[N_1 = \sums{k_c\div
  \Pi'(\delta,\cc)\\ \cp{k_c}{\al_0\Pi(\aa)\Pi(\bb)}} \mu(k_c) N_1(k_c),\]
where \[N_1(k_c) = \# \left\{\beta_0 \in \ZZ \Where
  \begin{aligned}
    &\congr{\al_0^{a_0}\Pi(\aa)}{-\beta_0^{b_0}\Pi(\bb)} {k_c\Pi(\cc)}\\
    &\text{(\ref{eq:gen_cpb}), $h(\al_0,\beta_0,\aa,\bb,\cc,\delta;B) \le 1$}
  \end{aligned}
\right\}.\]

We note that both $\al_0^{a_0}\Pi(\aa)$ and $\Pi(\bb)$ are coprime to
$k_c\Pi(\cc)$.  Indeed, we have $\cp{k_c}{\al_0\Pi(\aa)\Pi(\bb)}$ by the
restriction on $k_c$ just introduced, and
$\cp{\Pi(\cc)}{\al_0\Pi(\aa)\Pi(\bb)}$ by (\ref{eq:gen_cpa}) and
(\ref{eq:gen_cpabc}).

We choose integers $A_1,A_2$ resp.\ $B_1,B_2$ depending only on $\al_0,\aa$
resp.\ $\bb$ such that
\begin{equation}\label{eq:choice_AB}
A_1A_2^{b_0} = \al_0^{a_0}\Pi(\aa), \quad B_1B_2^{b_0} = \Pi(\bb).
\end{equation}

For example, \[A_1 = \al_0^{a_0}\Pi(\aa),\quad A_2 = 1,\quad B_1 =
\Pi(\bb),\quad B_2 = 1\] is one valid choice. Often it turns out to be
convenient to move coordinates to $A_2$ that occur to a power of $b_0$ in
$\al_0^{a_0}\Pi(\aa)$; similarly for $B_2$.

Then $A_1,A_2,B_1,B_2$ are coprime to $k_c\Pi(\cc)$.  For each $\beta_0$
satisfying \[\congr{\al_0^{a_0}\Pi(\aa)}{-\beta_0^{b_0}\Pi(\bb)}{k_c\Pi(\cc)}\]
there is a unique $\rho \in \{1, \dots, k_c\Pi(\cc)\}$ satisfying
\begin{equation}\label{eq:generic_rho}
  \cp{\rho}{k_c\Pi(\cc)}, \quad 
  \congr{A_1}{-\rho^{b_0}B_1}{k_c\Pi(\cc)}
\end{equation}
and \[\congr{\beta_0 B_2}{\rho A_2}{k_c\Pi(\cc)}.\]

This shows that \[N_1(k_c) = \sum_{\substack{1 \le \rho \le
    k_c\Pi(\cc)\\\text{(\ref{eq:generic_rho}) holds}}}
\#\left\{\beta_0 \in \ZZ \WHERE
  \begin{aligned}
    &\congr{\beta_0 B_2}{\rho A_2}{k_c\Pi(\cc)}\\
    &\text{(\ref{eq:gen_cpb}), $h(\al_0,\beta_0,\aa,\bb,\cc,\delta;B) \le 1$}
  \end{aligned}
\right\}\]

We remove the coprimality condition~(\ref{eq:gen_cpb}) on $\beta_0$ by another
M\"obius inversion; writing $\beta_0 = k_b\beta_0'$, we get
\[N_1(k_c) = \sums{1 \le \rho \le
    k_c\Pi(\cc)\\\text{(\ref{eq:generic_rho}) holds}}
\sum_{k_b \div \Pi'(\delta,\bb)\Pi(\aa)}
\mu(k_b) N_1(\rho,k_b,k_c)\]
with \[N_1(\rho,k_b,k_c) = \#\left\{\beta_0' \in \ZZ \Where
  \begin{aligned}
    &\congr{k_b\beta_0'B_2}{\rho A_2}{k_c\Pi(\cc)}\\
    &h(\al_0,k_b\beta_0',\aa,\bb,\cc,\delta;B) \le 1
  \end{aligned}
\right\}.\] Here, we may restrict to $k_b$ satisfying $\cp{k_b}{k_c\Pi(\cc)}$
because otherwise $\cp{\rho A_2}{k_c\Pi(\cc)}$ implies that
$N_1(\rho,k_b,k_c)=0$.  We note that we have
$\cp{k_bB_2}{k_c\Pi(\cc)}$ after this restriction.

We recall that $\{t \in \RR \where h(\al_0, t, \aa, \bb, \cc, \delta;B) \le
1\}$ is assumed to consist of intervals $I_1, \dots, I_n$, with $I_j$ starting
at $t_{0,j}$ and ending at $t_{1,j}$. Let $\psi(t) = \{t\} - 1/2$, where
$\{t\}$ is the fractional part of $t \in \RR$. For $j = 1, \dots, n$, by
\cite[Lemma~3]{MR2320172},
\begin{multline*}
  \#\left\{\beta_0' \in \ZZ \Where
    \begin{aligned}
      &\congr{k_b \beta_0'B_2}{\rho A_2}{k_c\Pi(\cc)},\\
      &k_b\beta_0' \in I_j
    \end{aligned}
  \right\}\\ = \frac{t_{1,j}-t_{0,j}}{k_b k_c\Pi(\cc)} +
  \psi\left(\frac{k_b^{-1} t_{0,j} -\rho A_2 \overline{k_b
        B_2}}{k_c\Pi(\cc)}\right) - \psi\left(\frac{k_b^{-1} t_{1,j} - \rho
      A_2\overline{k_b B_2}}{k_c\Pi(\cc)}\right),
\end{multline*}
where $t_{0,j}, t_{1,j}$ (depending on $\al_0,\aa,\bb,\cc,\delta$ and $B$) are
the start and end points of $I_j$, and $\overline{x}$ is the multiplicative
inverse modulo $k_c\Pi(\cc)$ of an integer $x$ coprime to $k_c\Pi(\cc)$.

We define
\begin{equation}\label{eq:generic_V0}
  V_1(\al_0,\aa,\bb,\cc,\delta;B) = 
  \int_{h(\al_0,t,\aa,\bb,\cc,\delta;B) \le 1} \frac{1}{\Pi(\cc)} \dd t.
\end{equation}
The sum of the lengths of the intervals $I_1, \dots, I_n$ is
$\Pi(\cc)V_1(\al_0,\aa,\bb,\cc,\delta;B)$, so
\[N_1(\rho,k_b,k_c) = \frac{1}{k_bk_c} V_1(\al_0,\aa,\bb,\cc,\delta;B)
+ R_1(\rho,k_b,k_c),\] with
\begin{equation*}
  R_1(\rho,k_b,k_c) = \sum_{j=1}^n \sum_{i \in \{0,1\}}
  (-1)^i \psi\left(\frac{k_b^{-1} t_{i,j} -\rho A_2 \overline{k_b B_2}}
    {k_c\Pi(\cc)}\right)
\end{equation*}

Tracing through the argument gives the following estimation for $N_{U,H}(B)$,
where, for any $n \in \ZZp$, $\phis(n)=\frac{\phi(n)}{n} =
\prodpd{n}{\left(1-1/p\right)}$ and $\omega(n)$ is the number of distinct prime
factors of $n$.

\begin{prop}\label{prop:generic_first_step}
  If the counting problem has the special form described at the beginning of
  this section, then
  \begin{equation*}
    N_{U,H}(B) = \sums{(\al_0,\aa,\bb,\cc,\delta) \in \ZZ_* \times
      \ZZp^{r+s+t+1}\\\text{(\ref{eq:gen_cpa}), (\ref{eq:gen_cpabc}) holds}}
    N_1, 
  \end{equation*}
  with \[N_1 = \theta_1(\al_0, \aa, \bb, \cc, \delta)
  V_1(\al_0,\aa,\bb,\cc,\delta;B) + R_1(\al_0, \aa, \bb, \cc, \delta;B),\]
  where $V_1$ is defined by (\ref{eq:generic_V0}) and, with
  $A_1,A_2,B_1,B_2$ as in (\ref{eq:choice_AB}),
  \begin{multline*}
    \theta_1(\al_0, \aa, \bb, \cc, \delta)\\ = \sums{k_c\div
      \Pi'(\delta,\cc)\\\cp{k_c}{\al_0\Pi(\aa)\Pi(\bb)}}
    \frac{\mu(k_c)\phis(\Pi'(\delta,\bb)\Pi(\aa))}{k_c\phis(\gcd(\Pi'(\delta,\bb),k_c\Pi(\cc)))}
    \sums{1 \le \rho \le
      k_c\Pi(\cc)\\\text{(\ref{eq:generic_rho}) holds}} 1
  \end{multline*}
  and
  \begin{multline*}
    R_1(\al_0,\aa,\bb,\cc,\delta;B) = \sums{k_c \div
      \Pi'(\delta,\cc)\\\cp{k_c}{\al_0\Pi(\aa)\Pi(\bb)}} \mu(k_c)
    \sums{k_b \div \Pi'(\delta,\bb)\Pi(\aa)\\\cp{k_b}{k_c\Pi(\cc)}} \mu(k_b)\\
    \times \sums{1 \le \rho \le k_c\Pi(\cc)\\\text{(\ref{eq:generic_rho})
        holds}}
    \sum_{j=1}^n \sum_{i \in \{0,1\}} (-1)^i \psi\left(\frac{k_b^{-1}
        t_{i,j} -\rho A_2 \overline{k_b B_2}} {k_c\Pi(\cc)}\right).
  \end{multline*}

  We have $R_1(\al_0, \aa, \bb, \cc, \delta;B)=0$ if
  $h(\al_0,t,\aa,\bb,\cc,\delta;B)>1$ for all $t\in \RR$, while
  \[R_1(\al_0, \aa, \bb, \cc, \delta;B) \\\ll
  2^{\omega(\Pi'(\delta,\cc))}2^{\omega(\Pi'(\delta,\bb)\Pi(\aa))}
  b_0^{\omega(\delta\Pi(\cc))}\] otherwise.
\end{prop}

\begin{proof}
  For the main term, we note that $\theta_1$ is \begin{multline*}
    \sums{k_c \div \Pi'(\delta,\cc)\\
      \cp{k_c}{\al_0\Pi(\aa)\Pi(\bb)}} \frac{\mu(k_c)}{k_c} \sums{1 \le \rho
      \le k_c\Pi(\cc)\\\text{(\ref{eq:generic_rho}) holds}}
    \sums{k_b\div \Pi'(\delta,\bb)\Pi(\aa)\\
      \cp{k_b}{k_c\Pi(\cc)}} \frac{\mu(k_b)}{k_b}\\
    = \sums{k_c \div \Pi'(\delta,\cc)\\ \cp{k_c}{\al_0\Pi(\aa)\Pi(\bb)}}
    \frac{\mu(k_c)\phis(\Pi'(\delta,\bb)\Pi(\aa))}
    {k_c\phis(\gcd(\Pi'(\delta,\bb)\Pi(\aa),k_c\Pi(\cc)))}
    \sum_{\substack{1 \le \rho \le k_c\Pi(\cc)\\\text{(\ref{eq:generic_rho})
          holds}}} 1
  \end{multline*} and use $\cp{\Pi(\aa)}{k_c\Pi(\cc)}$ by
  (\ref{eq:gen_cpabc}) and the assumption on $k_c$.
  
  Our discussion before the statement of this result immediately gives the
  explicit formula for the error term $R_1$. Additionally, we note that both
  $N_1$ and $V_1$ vanish if $h(\al_0,t,\aa,\bb,\cc,\delta)>1$ for all $t \in
  \RR$. Otherwise, we estimate the inner sums over $j,i$ by $O(1)$. The total
  error is
  \begin{equation*}
    \begin{split}
      \ll{}& \sum_{k_c \div \Pi'(\delta,\cc)} |\mu(k_c)| \sum_{k_b\div
        \Pi'(\delta,\bb)\Pi(\aa)} |\mu(k_b)| b_0^{\omega(k_c\Pi(\cc))}\\
      \ll{}& 2^{\omega(\Pi'(\delta,\cc))} 2^{\omega(\Pi'(\delta,\bb)\Pi(\aa))}
      b_0^{\omega(\delta\Pi(\cc))}.
    \end{split}
  \end{equation*}
  since (\ref{eq:generic_rho}) has at most $b_0^{\omega(k_c\Pi(\cc))}$
  solutions $\rho$ with $1 \le \rho \le k_c\Pi(\cc)$.
\end{proof}

In this estimation of $N_1$, we expect that $\theta_1V_1$ is the main term and
$R_1$ is the error term. It is sometimes possible (see
Lemma~\ref{lem:sum_a1a3_first} for an example) to show that the crude bound
for $R_1$ at the end of Proposition~\ref{prop:generic_first_step} summed over
all $\al_0,\aa,\bb,\cc, \delta$ for which there is a $t \in \RR$ with
$h(\al_0,t,\aa,\bb,\cc,\delta;B) \le 1$ gives a total contribution of
$o(B(\log B)^{k-1})$. In other cases, this is impossible, and one has to show
that there is additional cancellation when summing the precise expression for
$R_1$ of Proposition~\ref{prop:generic_first_step} over the remaining
variables (see \cite{arXiv:0807.4733}, for example).

\section{Another summation}\label{sec:generic_errors}

As the main result of this section, we show under certain conditions how to
sum an expression such as the main term of
Proposition~\ref{prop:generic_first_step} over another coordinate
(Proposition~\ref{prop:complete_summation_1} and
Proposition~\ref{prop:complete_summation_2}).

In this section, we will start to define several sets $\Theta_i$ of
real-valued functions in one variable and, for any $r \in \ZZp$, several sets
$\Theta_{j,r}$ and $\Theta_{j,r}'$ of real-valued functions in $r$ variables.
We will be interested in the average order of these functions when summed over
intervals.

Figure~\ref{fig:functions} gives an overview of the relations between these
sets of functions, for appropriate constants $C, C', C'', C_1, C_2, C_3 \in
\RRnn$ and $b \in \ZZp$, where each arrow denotes an inclusion. In case of an
arrow from a set $\Theta_{j,r}$ to a set $\Theta_i$, we regard the functions
in the first set as functions in one of the variables.

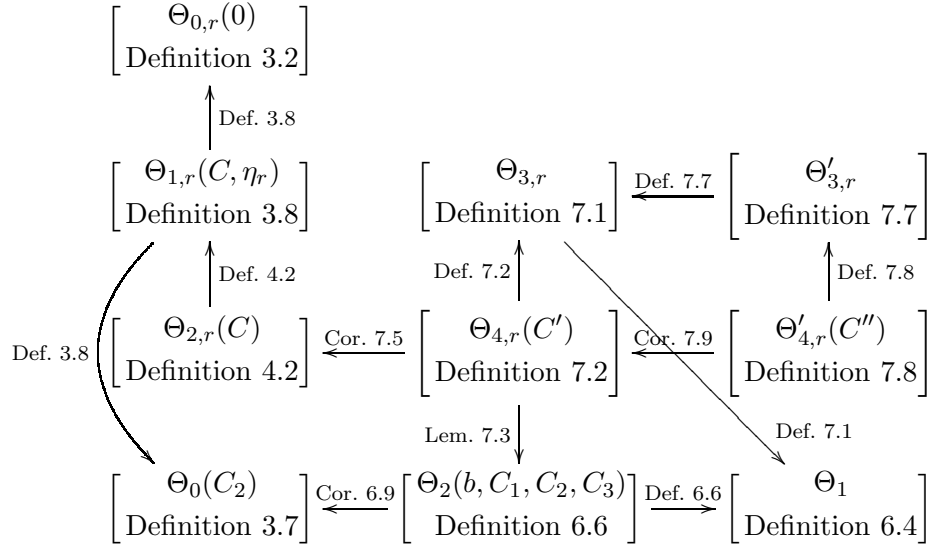
\begin{figure}[ht]
\begin{equation*}
  \xymatrix{
    \Thetabox{\Theta_{0,r}(0)}{\ref{def:function_bound}}
    &
    &
    \\
    \Thetabox{\Theta_{1,r}(C,\e_r)}{\ref{def:arithmetic_function_r_with_average}}
    \injref{u}{Def.~\ref{def:arithmetic_function_r_with_average}}
    \ar @/_3.5pc/[dd]_{\text{Def.~\ref{def:arithmetic_function_r_with_average}}}
    &
    \Thetabox{\Theta_{3,r}}{\ref{def:arithmetic_function_r}}
    \ar@{->}[ddr]^(0.80){\text{Def.~\ref{def:arithmetic_function_r}}}
    &
    \Thetabox{\Theta'_{3,r}}{\ref{def:arithmetic_function_r_simple}}
    \injref{l}{Def.~\ref{def:arithmetic_function_r_simple}}
    \\
    \Thetabox{\Theta_{2,r}(C)}{\ref{def:arithmetic_function_r_summable}}
    \injrefff{u}{Def.~\ref{def:arithmetic_function_r_summable}}
    &
    \Thetabox{\Theta_{4,r}(C')}{\ref{def:arithmetic_function_r_small}}
    \injref{l}{Cor.~\ref{cor:arithmetic_function_r_small_has_average}}
    \injreff{u}{Def.~\ref{def:arithmetic_function_r_small}}
    \injref{d}{Lem.~\ref{lem:arithmetic_function_r_small_is_small}}
    &
    \Thetabox{\Theta'_{4,r}(C'')}{\ref{def:arithmetic_function_r_simple_small}}
    \injref{l}{Cor.~\ref{cor:arithmetic_function_r_simple_small_is_small}}
    \injref{u}{Def.~\ref{def:arithmetic_function_r_simple_small}}
    \\
    \Thetabox{\Theta_0(C_2)}{\ref{def:arithmetic_function_with_average}}
    &
    \Thetabox{\Theta_2(b,C_1,C_2,C_3)}{\ref{def:function_small}}
    \injreff{r}{Def.~\ref{def:function_small}}
    \injref{l}{Cor.~\ref{cor:function_small_has_average}}
   &
    \Thetabox{\Theta_1}{\ref{def:function_product}}
  }
\end{equation*}
  \caption{Relations between our sets of functions}
  \label{fig:functions}
\end{figure}

\begin{lemma}\label{lem:sum_product_arithmetic_continuous}
  Let $\theta : \ZZ \to \RR$ be any function for which there exist $c \in
  \RRnn$ and a function $E: \RR \to \RR$ such that, for all $t \in \RRnn$,
  \begin{equation*}
    \sum_{0 < n \le t} \theta(n) = ct + E(t).
  \end{equation*}
  Let $t_1, t_2 \in \RRnn$, with $t_1 \le t_2$. Let $g:[t_1,t_2] \to \RR$ be a
  function that has a continuous derivative whose sign changes only $R(g)$
  times on $[t_1,t_2]$. Then
  \begin{multline*}
    \sum_{t_1 < n \le t_2} \theta(n)g(n)\\ = c \int_{t_1}^{t_2} g(t) \dd t +
    O\left((R(g)+1)\left(\sup_{t_1 \le t \le t_2} |E(t)|\right) \left(\sup_{t_1
          \le t \le t_2} |g(t)|\right)\right).
  \end{multline*}
\end{lemma}

\begin{proof}
  The proof is similar to \cite[Lemma~2]{arXiv:0710.1560}.  For any $t \in
  \RRnn$, let \[M(t) = \sum_{0 < n\le t} \theta(n),\qquad S(t_1,t_2) =
  \sum_{t_1 \le n \le t_2} \theta(n)g(n).\] Using partial summation, the estimate
  for $M(t)$ and integration by parts, $S(t_1,t_2)$ is 
  \begin{equation*}
    \begin{split}
       & M(t_2)g(t_2) - M(t_1)g(t_1) - \int_{t_1}^{t_2} M(t) g'(t)
      \dd t\\
      ={}& c \int_{t_1}^{t_2} g(t) \dd t +
      E(t_2)g(t_2)-E(t_1)g(t_1)-\int_{t_1}^{t_2} E(t) g'(t) \dd t\\
      ={}& c \int_{t_1}^{t_2} g(t) \dd t + O\left(\left(\sup_{t_1 \le t \le t_2}
          |E(t)|\right)\left(|g(t_1)|+|g(t_2)|+\int_{t_1}^{t_2} |g'(t)| \dd
          t\right)\right).
    \end{split}
  \end{equation*}
  The result follows once we split $[t_1,t_2]$ into $R(g)+1$ intervals
  where the sign of $g'$ does not change.
\end{proof}

\begin{definition}\label{def:function_bound}
  Let $C \in \RRnn$. Let $\Theta_{0,0}(C)$ be the set $\RR$ of real
  numbers. For any $r \in \ZZp$, we define $\Theta_{0,r}(C)$ recursively as
  the set of all non-negative functions $\theta: \ZZp^{r} \to \RR$ with the
  following property. For any $i \in \{1, \dots, r\}$, there is $\theta_i \in
  \Theta_{0,r-1}(C)$ such that, for any $t \in \RRnn$,
 \[\sum_{0 < \e_i \le t} \theta(\e_1, \dots,
  \e_r) \le \theta_i(\e_1, \dots, \e_{i-1}, \e_{i+1},\dots,\e_r)\cdot t(\log
  (t+2))^C.\]

  For any $\theta \in \Theta_{0,r}(C)$ and $i=1, \dots, r$, we fix a function
  $\theta_i \in \Theta_{0,r-1}(C)$ as above and denote it by $\M(\theta(\e_1,
  \dots, \e_r),\e_i)$. For any pairwise distinct $i_1, \dots, i_n \in \{1,
  \dots, r\}$, let
  \begin{multline*}
    \M(\theta(\e_1, \dots, \e_r),
    \e_{i_1}, \dots, \e_{i_n})\\ = \M(\dots\M(\theta(\e_1, \dots,
    \e_r),\e_{i_1})\dots,\e_{i_n}) \in \Theta_{0,r-n}(C).
  \end{multline*}
  For any $t \in \RRnn$, we have
  \begin{equation*}
    \sum_{0 < \e_{i_1}, \dots, \e_{i_n} \le t} \theta(\e_1, \dots, \e_r)
    \le \M(\theta(\e_1, \dots, \e_r),\e_{i_1},\dots,\e_{i_n}) t^n(\log(t+2))^{nC}.
  \end{equation*}
\end{definition}

\begin{example}\label{exa:phis_phid_omega}
  For any $n \in \ZZp$, let \[\phis(n)=\frac{\phi(n)}{n} =
  \prodpdp{n}{1-\frac 1 p}, \qquad
  \phid(n)= \prodpdp{n}{1+\frac 1 p}.\]

  Let $C \in \ZZnn$. For any $t \in \RRnn$, we have
  \begin{equation*}
    \sum_{0 < n \le t} (\phis(n))^C \le \sum_{0 < n \le t} (\phid(n))^C \ll_C t
  \end{equation*}
  (cf. \cite[Equation~3.1]{arXiv:0710.1560}) and
  \begin{equation*}
    \sum_{0 < n \le t} (1+C)^{\omega(n)} \ll_C t(\log(t+2))^C
  \end{equation*}
  (cf.\ \cite[Section~5.1]{MR2320172}).

  Therefore, for any $C \in \ZZnn$ and $r \in
  \ZZp$,
  \begin{equation*}
    \prod_{i=1}^r(\phis(\e_i))^C \in \Theta_{0,r}(0),\quad \prod_{i=1}^r
    (\phid(\e_i))^C \in \Theta_{0,r}(0),\quad
    \prod_{i=1}^r(1+C)^{\omega(\e_i)} \in \Theta_{0,r}(C).
  \end{equation*}
\end{example}

\begin{lemma}\label{lem:sum_over_n}
  Let $C \in \RRnn$. Let $\theta:\ZZ \to \RR$ be a non-negative function such
  that, for any $t \in \RRnn$, we have $\sum_{0 < n \le t} \theta(n) \le t
  (\log(t+2))^{C}$.

  Let $t_1 \le t_2 \in \RRnn$, $\kappa \in \RR$. Then
  \begin{equation*}
    \sum_{t_1 < n \le t_2} \frac{\theta(n)}{n^\kappa} \ll_{C,\kappa}
    \begin{cases}
      t_2^{1-\kappa}(\log(t_2+2))^{C}, &\kappa<1,\\
      (\log(t_2+2))^{C+1}, &\kappa=1,\\
      \frac{(\log(t_1+2))^{C}}{t_1^{\kappa-1}} \ll_{C,\kappa} 1, &\kappa>1.
    \end{cases}
\end{equation*}
\end{lemma}

\begin{proof}
  Let $S$ be the sum that we want to estimate. Let $M(t)=\sum_{0 < n \le t}
  \theta(n)$.

  By partial summation, 
  \begin{multline*}
    S = \frac{M(t_2)}{t_2^\kappa} - \frac{M(t_1)}{t_1^\kappa} -
    \int_{t_1}^{t_2} (-\kappa)\frac{M(t)}{t^{\kappa+1}} \dd t\\ \ll_{\kappa}
    \frac{(\log(t_2+2))^{C}}{t_2^{\kappa-1}} + \frac{(\log
      (t_1+2))^{C}}{t_1^{\kappa-1}} + \int_{t_1}^{t_2}
    \frac{(\log(t+2))^{C}}{t^\kappa} \dd t.
  \end{multline*}
  If $\kappa=1$, the result follows from \[\int_{t_1}^{t_2}
  \frac{(\log(t+2))^{C}}{t} \dd t =
  \frac{(\log(t_2+2))^{C+1}-(\log(t_1+2))^{C+1}}{C+1}.\] For $\kappa \ne
  1$, the result follows by induction over $C$ from
  \begin{multline*}
    \int_{t_1}^{t_2} \frac{(\log(t+1))^{C}}{t^\kappa} \dd t\\
    \ll_{C,\kappa} \frac{(\log(t_2+1))^{C}}{t_2^{\kappa-1}} +
    \frac{(\log(t_1+1))^{C}}{t_1^{\kappa-1}} + \int_{t_1}^{t_2}
    \frac{(\log(t+1))^{C-1}}{t^\kappa} \dd t,
  \end{multline*}
  which is obtained using integration by parts. Depending on whether
  $\kappa<1$ or $\kappa>1$, the first or second term gives the main
  contribution.
\end{proof}

Now we come to the setup for the main result of this section. Let $r,s \in
\ZZnn$. We consider a non-negative function $V: \RRnn^{r+s+1} \times \RR_{\ge
  3} \to \RR$ with the following properties. We assume that, for $j =1, \dots,
s$, there are
\begin{equation*}
  k_{0,j}, \dots, k_{r+j-1,j} \in \RR,\quad k_{r+j,j} \in \RRnz,\quad k_{r+j+1,j},
  \dots, k_{r+s,j} = 0, \quad a_j \in \RRp
\end{equation*}
such that
\begin{equation}\label{eq:generic_Vbound}
V(\e_0,\dots,\e_{r+s};B) \ll
\frac{B^{1-A}}{\e_0^{1-A_0}\cdots \e_{r+s}^{1-A_{r+s}}},
\end{equation}
where we define, for $i = 0, \dots, r+s$,
\[A=\sum_{j=1}^s a_j,\quad A_i=\sum_{j=1}^s a_jk_{i,j}.\]

We also assume that $V(\e_0,\dots,\e_{r+s};B)=0$ unless both
\begin{equation}\label{eq:height_s}
\e_0^{k_{0,j}}\cdots
\e_{r+s}^{k_{r+s,j}}=\e_0^{k_{0,j}}\cdots\e_{r+j}^{k_{r+j,j}} \le B,
\end{equation}
for $j=1, \dots, s$, and
\begin{equation}\label{eq:height_r}
  1 \le \e_i \le B,
\end{equation}
for $i=1, \dots, r$.

\begin{remark}\label{rem:implied_constants}
  In (\ref{eq:generic_Vbound}) and for the remainder of this section, we use
  the convention that all implied constants (in the notation $\ll$ and
  $O(\dots)$) are independent of $\e_0, \dots, \e_{r+s}$ and $B$, but may
  depend on all other parameters, in particular on $V$ and $\theta$.
\end{remark}

\begin{lemma}\label{lem:summation_V_theta_bounds}
  In the situation described above, let $\theta \in \Theta_{0,r+s+1}(C)$ for
  some $C \in \ZZnn$. Then
  \begin{multline*}
    \sum_{\e_1, \dots, \e_{r+s}}\theta(\e_0, \dots, \e_{r+s})V(\e_0, \dots,
    \e_{r+s};B)\\ \ll \e_0^{-1} \M(\theta(\e_0, \dots, \e_{r+s}),\e_{r+s},
    \dots, \e_1) B(\log B)^{r+(r+s)C}.
  \end{multline*}
\end{lemma}

\begin{proof}
  For any $\ell \in \{0, \dots, r+s-1\}$, let \[\theta_\ell(\e_0, \dots,
  \e_\ell) = \M(\theta(\e_0, \dots, \e_{r+s}),\e_{r+s}, \dots, \e_{\ell+1})
  \in \Theta_{0,\ell+1}(C).\]
  For $\ell=s, \dots, 0$, we claim that
  \begin{multline*}
    \sum_{\e_{r+\ell+1}, \dots, \e_{r+s}} \theta(\e_0, \dots, \e_{r+s})V(\e_0,
    \dots, \e_{r+s};B)\\ \ll \frac{\theta_{r+\ell}(\e_1, \dots,
      \e_{r+\ell})B^{1-A^{(\ell)}}(\log B)^{(s-\ell)C}}{\e_0^{1-A_0^{(\ell)}}\cdots
      \e_{r+\ell}^{1-A_{r+\ell}^{(\ell)}}},
  \end{multline*}
  where
  \begin{equation*}
    A^{(\ell)}=\sum_{j=1}^\ell a_j,\quad A_i^{(\ell)}= \sum_{j=1}^\ell a_jk_{i,j}.
  \end{equation*}
  For $\ell=s$, this is true by (\ref{eq:generic_Vbound}). To prove the claim
  in the other cases by induction, we must estimate
  \begin{equation}\label{eq:one_sum}
    \sum_{\e_{r+\ell}} \frac{\theta_{r+\ell}(\e_0, \dots,
      \e_{r+\ell})B^{1-A^{(\ell)}}(\log
      B)^{(s-\ell)C}}{\e_0^{1-A_0^{(\ell)}}\cdots
      \e_{r+\ell}^{1-A_{r+\ell}^{(\ell)}}},
  \end{equation}
  for $\ell=s, \dots, 1$. Since $V(\e_0, \dots, \e_{r+s};B)=0$ unless
  (\ref{eq:height_s}), the summation can be restricted to $\e_{r+\ell}$
  satisfying $\e_{r+\ell} \le T$ if $k_{r+\ell,\ell}>0$ resp.\ $\e_{r+\ell} \ge
  T$ if $k_{r+\ell,\ell}<0$, with $T = (B/(\e_0^{k_{0,\ell}}\cdots
  \e_{r+\ell-1}^{k_{r+\ell-1,\ell}}))^{1/k_{r+\ell,\ell}}$. An application of
  Lemma~\ref{lem:sum_over_n} (with $\kappa = 1-A_{r+\ell}^{(\ell)} = 1-a_\ell
  k_{r+\ell,\ell}$) shows that (\ref{eq:one_sum}) is
  \begin{equation*}
    \ll \frac{\theta_{r+\ell-1}(\e_0, \dots, \e_{r+\ell-1})
      B^{1-A^{(\ell)}+a_\ell}(\log B)^{(s-(\ell-1))C}}
    {\e_0^{1-A_0^{(\ell)}+a_\ell k_{0,\ell}}\cdots
      \e_{r+\ell-1}^{1-A^{(\ell)}_{r+\ell-1}+a_\ell k_{r+\ell-1,\ell}}}.
  \end{equation*}
  The induction step is completed by observing $A^{(\ell)}-a_\ell =
  A^{(\ell-1)}$ and $A^{(\ell)}_i-a_\ell k_{i,\ell} = A^{(\ell-1)}_i$, for
  $i=0, \dots, r+\ell-1$.
  
  For $\ell = r, \dots, 0$, we claim that
  \begin{multline*}
    \sum_{\e_{\ell+1}, \dots, \e_{r+s}} \theta(\e_0, \dots, \e_{r+s})V(\e_0,
    \dots, \e_{r+s};B)\\ \ll \frac{\theta_\ell(\e_0, \dots, \e_\ell)B(\log
      B)^{r-\ell+(r+s-\ell)C}}{\e_0\cdots\e_\ell}.
  \end{multline*}
  This is also proved by induction. The case $\ell=r$ is the ending of our
  first induction. From here, we apply Lemma~\ref{lem:sum_over_n} (with
  $\kappa=1$) for the summation over $\e_\ell$ subject to (\ref{eq:height_r}).
\end{proof}

\begin{definition}\label{def:arithmetic_function_with_average}
  For any $C \in \RRnn$, let $\Theta_0(C)$ be the set of all non-negative
  functions $\theta: \ZZp \to \RR$ such that there is a $c_0 \in \RRnn$ and a
  bounded function $E: \RRnn \to \RR$ such that, for any $t \in
  \RRnn$, \[\sum_{0 < n \le t} \theta(n) = c_0 t+E(t)(\log(t+2))^C.\]
  
  If $\theta \in \Theta_0(C)$, the corresponding $c_0, E(t)$ are unique since
  $t$ grows faster than any power of $\log(t+2)$ for large $t$; we introduce
  the notation \[\A(\theta(n),n) = c_0, \qquad \E(\theta(n),n) =
  \sup_{t \in \RRnn}\{|E(t)|\}.\]
\end{definition}

\begin{definition}\label{def:arithmetic_function_r_with_average}
  For any $C \in \RRnn$ and $r \in \ZZp$, let $\Theta_{1,r}(C,\e_r)$ be the
  set of all functions $\theta: \ZZp^r \to \RR$ in the variables $\e_1, \dots,
  \e_r$ such that
  \begin{enumerate}[(1)]
  \item $\theta(\e_1, \dots, \e_r)$ as a function in $\e_1, \dots, \e_r$ lies
    in $\Theta_{0,r}(0)$.
  \item $\theta(\e_1, \dots, \e_r)$ as a function in $\e_r$ lies in
    $\Theta_0(C)$ for any $\e_1, \dots, \e_{r-1} \in \ZZ$, so that we have
    corresponding
    \[\A(\theta(\e_1, \dots , \e_r),\e_r) : \ZZp^{r-1} \to \RR, \quad
    \E(\theta(\e_1, \dots, ,\e_r), \e_r) : \ZZp^{r-1} \to \RR\] as functions in 
    $\e_1, \dots, \e_{r-1}$.
  \item $\A(\theta(\e_1, \dots , \e_r),\e_r)$ lies in $\Theta_{0,r-1}(0)$.
  \item $\E(\theta(\e_1, \dots, ,\e_r), \e_r)$ lies in $\Theta_{0,r-1}(C)$.
  \end{enumerate}
  We define $\Theta_{1,r}(C,\e_i)$ for any other variable $\e_i$ analogously.
\end{definition}

We want to estimate
\[\sum_{\e_0} \theta(\e_0, \dots, \e_{r+s})V(\e_0,\dots,\e_{r+s};B).\]
We assume that $V$ is as described before
Lemma~\ref{lem:summation_V_theta_bounds} with the additional property that $V$
as a function in the first variable $\eta_0$ has a continuous derivative whose
sign changes only finitely often on the interval $[1,B]$ and vanishes outside
this interval.

\begin{prop}\label{prop:complete_summation_1}
  Let $V$ be as above, and let $\theta \in \Theta_{1,r+s+1}(C,\e_0)$ for some
  $C \in \RRnn$. Then
  \begin{multline*}
    \sum_{\e_0} \theta(\e_0, \dots, \e_{r+s}) V(\e_0,\dots,\e_{r+s};B)=\\
    \A(\theta(\e_0, \dots, \e_{r+s}),\e_0) \int_{t_0 \ge 1} V(t_0,\e_1, \dots,
    \e_{r+s};B) \dd t_0 + R(\e_1, \dots, \e_{r+s};B),
  \end{multline*}
  where
  \begin{equation*}
    \sum_{\e_1, \dots,
      \e_{r+s}} R(\e_1,\dots,\e_{r+s};B) \ll B(\log B)^r(\log \log
    B)^{\max\{1,s\}}.
  \end{equation*}
\end{prop}

\begin{proof}
  We note that we may always assume that $1 \le \e_0, \dots, \e_r \le B$ since
  all terms and error terms vanish otherwise. Let $\theta' \in
  \Theta_{0,r+s}(0)$ and $\theta'' \in \Theta_{0,r+s}(C)$ be defined as
  \begin{equation*}
    \begin{split}
      \theta'(\e_1, \dots, \e_{r+s}) &= \A(\theta(\e_0, \dots, \e_{r+s}),\e_0),\\
      \theta''(\e_1, \dots, \e_{r+s}) &= \E(\theta(\e_0, \dots, \e_{r+s}),\e_0).
    \end{split}
  \end{equation*}
  We proceed in three steps. Let $T = (\log B)^{(r+s+1)C}$.
  \begin{enumerate}[(1)]
  \item\label{it:gen_step1} We show that 
    \begin{equation*}
      \sums{\e_0, \dots, \e_{r+s}\\\e_0 < T}\theta(\e_0, \dots, \e_{r+s}) V(\e_0,\dots,\e_{r+s};B)
      \ll B(\log B)^r(\log \log B).
    \end{equation*}
  \item\label{it:gen_step2} Combining $\theta \in \Theta_0(C)$ as a function
    in $\eta_0$ with Lemma~\ref{lem:sum_product_arithmetic_continuous}, we
    have
    \begin{equation*}
      \begin{split}
        &\sum_{\e_0 \ge T} \theta(\e_0, \dots, \e_{r+s}) V(\e_0,\dots,\e_{r+s};B)\\
        ={}& \theta'(\e_1,\dots,\e_{r+s})\int_{t_0 \ge T}
        V(t_0,\e_1,\dots,\e_{r+s};B) \dd t_0\\ &+ O\left(\theta''(\e_1, \dots,
          \e_{r+s})(\log B)^{C} \sup_{t_0\ge T} V(t_0,\e_1, \dots,
          \e_{r+s};B)\right)
      \end{split}
    \end{equation*}
    Here, we show that summing the error term over $\e_1, \dots, \e_{r+s}$
    gives $O(B(\log B)^r)$.
  \item\label{it:gen_step3} To complete the proof, we must
    estimate \[\sum_{\e_1, \dots, \e_{r+s}} \theta'(\e_1, \dots,
    \e_{r+s})\int_{1}^{T} V(t_0,\e_1, \dots, \e_{r+s};B) \dd t_0.\] If $s=1$
    and $k_{0,1} > 0$, we consider the case
    $T^{k_{0,1}}\e_1^{k_{1,1}}\cdots\e_{r+1}^{k_{r+1,1}} \le B$ and its
    opposite separately. If $s > 1$, we distinguish $2^s$ cases.
  \end{enumerate}

  For \eqref{it:gen_step1}, we use $\theta \in \Theta_{0,r+s+1}(0)$ and 
  Lemma~\ref{lem:summation_V_theta_bounds} for the summation over $\e_1,
  \dots, \e_{r+s}$ and Lemma~\ref{lem:sum_over_n} for the summation over
  $\e_0$ to compute
  \begin{equation*}
    \begin{split}
      &\sum_{\e_0, \dots, \e_{r+s}} \theta(\e_0, \dots, \e_{r+s})V(\e_0,
      \dots,
      \e_{r+s};B)\\
      \ll{}& \sum_{1 \le \e_0 < T} \e_0^{-1}\M(\theta(\e_0, \dots,
      \e_{r+s}),\e_{r+s}, \dots, \e_1)B(\log B)^r\\
      \ll{}& B(\log B)^r(\log \log B).
    \end{split}
  \end{equation*}

  For \eqref{it:gen_step2}, because of~(\ref{eq:height_s}), the error term
  vanishes unless, for $j=1, \dots,
  s$, \[T^{k_{0,j}}\e_1^{k_{1,j}}\cdots\e_{r+s}^{k_{r+s,j}} \le B.\]

  If $A_0\le 1$, using $\theta'' \in \Theta_{0,r+s}(C)$ and
  Lemma~\ref{lem:summation_V_theta_bounds} (with $\e_0=T$), we compute
  \[\begin{split} &\sum_{\e_1, \dots, \e_{r+s}} (\log
    B)^{C}\theta''(\e_1, \dots, \e_{r+s})
    \sup_{t_0 \ge T} V(t_0, \e_1, \dots, \e_{r+s};B)\\
    \ll{}&\sum_{\e_1, \dots, \e_{r+s}} \frac{(\log B)^C\theta''(\e_1, \dots,
      \e_{r+s})B^{1-A}}{T^{1-A_0}\e_1^{1-A_1}\cdots
      \e_{r+s}^{1-A_{r+s}}}\\
    \ll{}& T^{-1} B(\log B)^{r+(r+s+1)C}\\
    \ll{}& B(\log B)^r.
  \end{split}\]

  If $A_0 > 1$, then (\ref{eq:height_s}) implies that $V(t_0,\e_1, \dots,
  \e_{r+s};B)=0$ unless \[t_0^{A_0}\e_1^{A_1}\cdots\e_{r+s}^{A_{r+s}} \le
  B^A.\]
  Therefore, 
  \begin{equation*}
    \begin{split}
      &\sum_{\e_1, \dots, \e_{r+s}} (\log B)^{C}\theta''(\e_1, \dots, \e_{r+s})
      \sup_{t_0 \ge T} V(t_0, \e_1, \dots, \e_{r+s};B)\\
      \ll&{} \sum_{\e_1, \dots, \e_{r+s}} \frac{(\log B)^{C}\theta''(\e_1,
        \dots, \e_{r+s}) B^{1-A}}{\e_1^{1-A_1}\cdots\e_{r+s}^{1-A_{r+s}}}
      \sup_{T \le t_0 \le (B^A/(\e_1^{A_1}\cdots \e_{r+s}^{A_{r+s}}))^{1/A_0}}
      \frac{1}{t_0^{1-A_0}}\\
      \ll&{} \sum_{\e_1, \dots, \e_{r+s}} \frac{(\log B)^{C}\theta''(\e_1,
        \dots, \e_{r+s})
        B^{1-A/A_0}}{\e_1^{1-A_1/A_0}\cdots\e_{r+s}^{1-A_{r+s}/A_0}}
    \end{split}
  \end{equation*}
  We apply Lemma~\ref{lem:summation_V_theta_bounds} (with $\e_0=T$ and
  $k_{i,j}$ replaced by $k_{i,j}/A_0$) to conclude that this is $O(B(\log
  B)^r)$.

  For \eqref{it:gen_step3}, we assume $A_0=0$ first. We use $\theta' \in
  \Theta_{0,r+s}(0)$ and Lemma~\ref{lem:summation_V_theta_bounds} (with
  $\e_0=1$) to compute
  \[\begin{split} &\sum_{\e_1, \dots, \e_{r+s}} \theta'(\e_1, \dots,
    \e_{r+s})\int_1^T V(t_0,\e_1, \dots, \e_{r+s};B)
    \dd t_0\\
    \ll{}& \sum_{\e_1, \dots, \e_{r+s}} \frac{\theta'(\e_1, \dots,
      \e_{r+s})B^{1-A}}{\e_1^{1-A_1}\cdots
      \e_{r+s}^{1-A_{r+s}}} \int_1^T \frac{1}{t_0}\dd t_0\\
    \ll{}& B(\log B)^r(\log \log B).
  \end{split}\] 

  Now we suppose $A_0>0$. Let \[X_j=\e_1^{k_{1,j}} \cdots
  \e_{r+s}^{k_{r+s,j}}=\e_1^{k_{1,j}}\cdots \e_{r+j}^{k_{r+j,j}},\] for $j=1,
  \dots, s$.  We distinguish $2^s$ cases, labeled by the subsets $J$ of $\{1,
  \dots, s\}$.  In case $J$, we assume $X_j \le \min\{BT^{-k_{0,j}},B\}$ for
  each $j \in J$, and $X_j>\min\{BT^{-k_{0,j}},B\}$ for each $j \notin J$. By
  (\ref{eq:height_s}), $V(t_0, \e_1, \dots, \e_{r+s};B)=0$ unless
  $t_0^{k_{0,j}}X_j \le B$. Therefore, we may restrict to $X_j \le
  \max_{1 \le t_0 \le T} \{Bt_0^{-k_{0,j}}\}$.

  In total, in case $J$, we may restrict the summation over $\e_1, \dots,
  \e_{r+s}$ to
  \begin{equation*}
    X_j \in
    \begin{cases}
      [1,BT^{-k_{0,j}}], &j \in J,\ k_{0,j} \ge 0,\\
      (BT^{-k_{0,j}},B], &j \notin J,\ k_{0,j} \ge 0,\\
      [1,B], &j \in J,\ k_{0,j} < 0,\\
      (B,BT^{-k_{0,j}}], &j \notin J,\ k_{0,j} < 0;\\
    \end{cases}
  \end{equation*}
  in particular, the summation is trivial if $k_{0,j} = 0$ for some $j \notin
  J$, so we assume there is no such $j$.  Furthermore, we may restrict
  the integration over $t_0$ to the interval $[T_1,T_2]$ where
  \begin{equation*}
    T_1 = \maxs{j \in \{1,\dots, s\},\\k_{0,j} < 0}\{1,
    (BX_j^{-1})^{1/k_{0,j}}\}, \qquad T_2 = 
    \mins{j \in \{1, \dots, s\}\\k_{0,j} > 0}\{T,(BX_j^{-1})^{1/k_{0,j}}\};
  \end{equation*}
  we may assume that $T_1 \le T_2$ since the
  integral vanishes otherwise. We note that $1 \le (BX_j^{-1})^{1/k_{0,j}} \le
  T$ if and only if $j \notin J$.

  We define
  \begin{equation*}
    A'=\sum_{j \in J} a_j,\qquad A'_0=\sums{j \in J\\k_{0,j} > 0}
    a_jk_{0,j},\qquad A'_i=\sum_{j \in J} a_jk_{i,j},
  \end{equation*}
  for $i=1, \dots, r+s$.

  Combining (\ref{eq:generic_Vbound}) with
  \begin{multline*}
    \int_{T_1}^{T_2} \frac 1 {t_0^{1-A_0}} \dd t_0 \ll T_1^{A_0}+T_2^{A_0} \ll
    \prods{j \in J\\k_{0,j} > 0} T^{a_j k_{0,j}} \prod_{j \notin
      J}(BX_j^{-1})^{a_j}\\ =
    \frac{B^{A-A'}T^{A_0'}}{\e_1^{A_1-A_1'}\cdots\e_{r+s}^{A_{r+s}-A_{r+s}'}}, 
  \end{multline*}
  we obtain as the contribution of case $J$ to the error term of
  \eqref{it:gen_step3}
  \begin{equation*}
    \begin{split} &\sum_{\e_1, \dots, \e_{r+s}}
      \theta'(\e_1, \dots, \e_{r+s}) \int_1^T
      V(t_0, \e_1, \dots \e_{r+s};B) \dd t_0\\
      \ll{}& \sum_{\e_1, \dots, \e_{r+s}} \frac{\theta'(\e_1, \dots,
        \e_{r+s})B^{1-A}}{\e_1^{1-A_1}\cdots\e_{r+s}^{1-A_{r+s}}} \int_{T_1}^{T_2}
      \frac{1}{t_0^{1-A_0}} \dd t_0\\
      \ll{}& \sum_{\e_1, \dots, \e_{r+s}} \frac{\theta'(\e_1, \dots, \e_{r+s})
        B^{1-A'} T^{A'_0}}{\e_1^{1-A'_1} \cdots \e_{r+s}^{1-A'_{r+s}}}.
    \end{split}
  \end{equation*}
  For $j=s, \dots, 1$, we handle the summation over $\e_{r+j}$ using $\theta'
  \in \Theta_{0,r+s}(0)$ and Lemma~\ref{lem:sum_over_n}. After the summations
  over $\e_{r+s}, \dots, \e_{r+j+1}$ are done, the exponent of $\e_{r+j}$ in
  the denominator is $1-a_jk_{r+j,j}$ if $j \in J$ and it is $1$
  otherwise. For $j \in J$ and $k_{0,j} \ge 0$, we use $X_j \le
  BT^{-k_{0,j}}$, i.e.,
  \begin{equation*}
    \e_{r+j}^{a_jk_{r+j,j}} \le
    \frac{B^{a_j}T^{-a_jk_{0,j}}}{\e_1^{a_jk_{1,j}}\cdots
      \e_{r+j-1}^{a_jk_{r+j-1,j}}}.
  \end{equation*}
  For $j \in J$ and $k_{0,j} < 0$, we use $X_j \le B$, i.e.,
  \begin{equation*}
    \e_{r+j}^{a_jk_{r+j,j}} \le
    \frac{B^{a_j}}{\e_1^{a_jk_{1,j}}\cdots
      \e_{r+j-1}^{a_jk_{r+j-1,j}}}.
  \end{equation*}
  For $j \notin J$, we use that $BT^{-k_{0,j}} < X_j \le B$, for $k_{0,j} >
  0$, resp.\ $B < X_j \le BT^{-k_{0,j}}$, for $k_{0,j} < 0$, implies that, for
  $\e_1, \dots, \e_{r+j-1}$ fixed, there are $\ll T^{k_{0,j}}$ possibilities
  for $\e_{r+j}$, which shows that we pick up a factor $(\log \log B)$.

  It follows that we can continue our estimation as
  \begin{equation*}
    \begin{split}
      &\ll \sum_{\e_1, \dots, \e_r} \frac{\M(\theta'(\e_1, \dots,
        \e_{r+s}),\e_{r+s}, \dots, \e_{r+1})
        B(\log \log B)^{s-\#J}}{\e_1\cdots \e_r}\\
      &\ll B(\log B)^r(\log \log B)^s
    \end{split}
  \end{equation*}
  since $0 \le \#J \le s$.
\end{proof}

The next result is concerned with a similar situation as in
Proposition~\ref{prop:complete_summation_1}, with $r \in \ZZp$ and $s=1$.

Let $V: \RR^{r+2} \times \RR_{\ge 3} \to \RR$ be a non-negative function,
and \[k_0, \dots, k_r \in \RR, \quad k_{r+1} \in \RRnz, \quad a,b \in \RRp\]
such that
\begin{equation}\label{eq:generic_Vbound_double}
   V(\e_0, \dots, \e_{r+1};B) \ll
   \min\left\{\frac{B^{1-a}}{\e_0^{1-ak_0}\cdots\e_{r+1}^{1-ak_{r+1}}},
     \frac{B^{1+b}}{\e_0^{1+bk_0}\cdots\e_{r+1}^{1+bk_{r+1}}}\right\}.
\end{equation}
We assume that $V(\e_0, \dots, \e_{r+1};B) = 0$ unless, for $i=0, \dots, r+1$, 
\begin{equation}\label{eq:height_r_double}
  1 \le \e_i \le B.
\end{equation}
We assume that $V$ as a function in the first variable
$\eta_0$ has a continuous derivative whose sign changes only finitely often on
the interval $[1,B]$.

\begin{prop}\label{prop:complete_summation_2}
  For some $C \in \RRnn$, let $\theta \in \Theta_{1,r+2}(C,\e_0)$. Let $V$ be
  as above. Then
  \begin{multline*}
    \sum_{\e_0} \theta(\e_0, \dots, \e_{r+1})V(\e_0, \dots, \e_{r+1};B)\\
    = \theta'(\e_1, \dots, \e_{r+1})\int_{t_0 \ge 1} V(t_0,\e_1, \dots,
    \e_{r+s};B) \dd t_0+ R(\e_1, \dots, \e_{r+1};B),
  \end{multline*}
  where \[\sum_{\e_1, \dots,
    \e_{r+1}} R(\e_1, \dots, \e_{r+1};B) \ll B(\log B)^r(\log \log B).\]
\end{prop}

\begin{proof}
  We define $\theta' \in \Theta_{0,r+1}(0)$ and $\theta'' \in
  \Theta_{0,r+1}(C)$ as in the proof of
  Proposition~\ref{prop:complete_summation_1}.  Let \[M=M(\e_0, \dots,
  \e_{r+1};B) = \theta(\e_0, \dots, \e_{r+1})V(\e_0, \dots, \e_{r+1};B)\] and
  \begin{multline*}
    M'(t) = M'(t,\e_1 \dots, \e_{r+1};B)\\ = \theta'(\e_1, \dots,
    \e_{r+1})\int_{t_0 \ge t} V(t_0, \e_1, \dots, \e_{r+1}; B) \dd t_0.
  \end{multline*}
  We want to show that $M$ summed over all $\e_0 \in \ZZp$ agrees with $M'(1)$
  up to an acceptable error. We do this in three steps, where $T = (\log
  B)^{1+(r+2)C}$.
  \begin{enumerate}[(1)]
  \item\label{it:step1} We show that $M$ summed over all $\e_0$ agrees with
    $M$ summed over $\e_0 \ge T$ up to an acceptable error, by
    proving that \[\sums{\e_0, \dots, \e_{r+1}\\\e_0 < T} M \ll
    B(\log B)^r(\log\log B).\]
  \item\label{it:step2} We show that $M$ summed over $\e_0 \ge T$ gives
    $M'(T)$ up to an error of $R' = R'(\e_1, \dots, \e_{r+1};B)$ with
    $\sum_{\e_1, \dots, \e_{r+1}} R' \ll B(\log B)^{r}$.
  \item\label{it:step3} We show that $M'(T)$ summed over $\e_1, \dots, \e_{r+1}$
    agrees with $M'(1)$ up to an acceptable error, by proving that \[\sum_{\e_1,
      \dots, \e_{r+1}}(M'(1)-M'(T)) \ll B(\log B)^r(\log \log B).\] If $k_0<0$, we distinguish
    three cases, where $\e_1^{k_1}\cdots \e_{r+1}^{k_{r+1}}$ is at most $B$,
    or at least $BT^{-k_0}$, or between these two numbers.
  \end{enumerate}

  For \eqref{it:step1}, we use (\ref{eq:generic_Vbound_double}), $\theta \in
  \Theta_{0,r+2}(0)$ and (\ref{eq:height_r_double}). For
  $\e_0^{k_0}\cdots\e_{r+1}^{k_{r+1}} \le B$, we apply
  Lemma~\ref{lem:summation_V_theta_bounds} to compute
  \begin{equation*}
    \begin{split}
      \sum_{\e_0, \dots, \e_{r+1}} M &\ll \sum_{\e_0, \dots, \e_{r+1}}
      \frac{\theta(\e_0, \dots, \e_{r+1})B^{1-a}}
      {\e_0^{1-ak_0}\cdots\e_{r+1}^{1-ak_{r+1}}}\\
      &\ll \sum_{\e_0}
      \e_0^{-1}\M(\theta(\e_0, \dots, \e_{r+1}),\e_{r+1}, \dots, \e_1)B(\log
      B)^r\\
      &\ll B(\log B)^r(\log\log B).
    \end{split}
  \end{equation*}
  In the opposite case, by Lemma~\ref{lem:sum_over_n}, we have 
  \begin{equation*}
    \begin{split} \sum_{\e_0, \dots,
        \e_{r+1}} M &\ll \frac{\theta(\e_0, \dots,
        \e_{r+1})B^{1+b}}{\e_0^{1+bk_0}\cdots \e_{r+1}^{1+bk_{r+1}}}\\
      &\ll \sum_{\e_0, \dots, \e_r}
      \frac{\M(\theta(\e_0, \dots, \e_{r+1}),\e_{r+1})B}{\e_0\cdots\e_r}\\
      &\ll B(\log B)^r(\log\log B).
    \end{split}
  \end{equation*}
  
  For \eqref{it:step2}, we combine $\theta \in \Theta_0(C)$ as a function in
  $\eta_0$ with Lemma~\ref{lem:sum_product_arithmetic_continuous}. This shows
  that $M$ summed over $\e_0\ge T$ gives the main term $M'(T)$ as above and an
  error term which can be estimated (using $V(\e_0, \dots, \e_{r+1};B) \ll
  \frac{B}{\e_0\cdots\e_{r+1}}$ by (\ref{eq:generic_Vbound_double}), $\theta''
  \in \Theta_{0,r+1}(C)$, (\ref{eq:height_r_double}) and
  Lemma~\ref{lem:sum_over_n}) as
  \[\begin{split}  &\ll \sum_{\e_1, \dots, \e_{r+1}} (\log
    B)^{C}\theta''(\e_1, \dots, \e_{r+1})\sup_{t_0 \ge T} V(t_0,
    \e_1, \dots, \e_{r+1};B)\\
    &\ll \sum_{\e_1, \dots, \e_{r+1}} \frac{(\log B)^{C}\theta''(\e_1,
      \dots, \e_{r+1})B}
    {T \e_1\cdots\e_{r+1}}\\
    &\ll T^{-1} B(\log B)^{r+1+(r+2)C} = B(\log B)^r.
  \end{split}\]

  For \eqref{it:step3}, we suppose $k_{r+1}>0$; the case $k_{r+1}<0$ is similar.
  In the following computations, we use (\ref{eq:generic_Vbound_double}),
  $\theta' \in \Theta_{0,r+1}(0)$, (\ref{eq:height_r_double}) and
  Lemma~\ref{lem:sum_over_n}.

  If $k_0<0$, we split the summation over $\e_1, \dots, \e_{r+1}$ and
  integration over $t_0$ into three parts, the first defined by the condition
  $\e_1^{k_1}\cdots \e_{r+1}^{k_{r+1}}\le B$. We
  estimate using Lemma~\ref{lem:summation_V_theta_bounds} (with $\e_0=1$)
  \begin{equation*}
    \begin{split}
      \ll{}& \sum_{\e_1, \dots, \e_{r+1}} \theta'(\e_1, \dots, \e_{r+1})
      \int_{1}^{T}
      V(t_0,\e_1, \dots, \e_{r+1};B) \dd t_0\\
      \ll{}&\sum_{\e_1, \dots, \e_{r+1}} \theta'(\e_1, \dots, \e_{r+1})
      \int_{1}^{T}
      \frac{B^{1-a}}{t_0^{1-ak_0}\e_1^{1-ak_1}\cdots\e_{r+1}^{1-ak_{r+1}}}
      \dd t_0\\
      \ll{}&\sum_{\e_1, \dots, \e_{r+1}} \frac{\theta'(\e_1, \dots,
        \e_{r+1})B^{1-a}}{\e_1^{1-ak_1}\cdots\e_{r+1}^{1-ak_{r+1}}}\\
      \ll{}& B(\log B)^r.
    \end{split}
  \end{equation*}
  For the second subset defined by $B < \e_1^{k_1}\cdots
  \e_{r+1}^{k_{r+1}} \le BT^{-k_0}$, we get \[\begin{split} \ll{}&
    \sum_{\e_1, \dots, \e_{r+1}}
    \theta'(\e_1, \dots, \e_{r+1})\\
    &\times \left(\int_{t_0 \le (\e_1^{k_1}\cdots
        \e_{r+1}^{k_{r+1}}/B)^{-1/k_0}}
      \frac{B^{1+b}}{t_0^{1+bk_0}\e_1^{1+bk_1}\cdots\e_{r+1}^{1+bk_{r+1}}} \dd
      t_0\right.\\ &\qquad \left.+ \int_{t_0 \ge (\e_1^{k_1}\cdots
        \e_{r+1}^{k_{r+1}}/B)^{-1/k_0}} \frac{B^{1-a}}
      {t_0^{1-ak_0}\e_1^{1-ak_1}\cdots\e_{r+1}^{1-ak_{r+1}}}
      \dd t_0\right)\\
    \ll{}&\sum_{\e_1, \dots, \e_{r+1}} \frac{\theta'(\e_1, \dots, \e_{r+1})B}{\e_1\cdots\e_{r+1}}\\
    \ll{}& B(\log B)^r(\log \log B).
  \end{split}\] For the third subset defined by $\e_1^{k_1}\cdots
  \e_{r+1}^{k_{r+1}} > BT^{-k_0}$, we get \[\begin{split} &\ll \sum_{\e_1,
      \dots, \e_{r+1}} \int_{1}^{T}
    \frac{\theta'(\e_1,\dots,\e_{r+1})B^{1+b}}
    {t_0^{1+bk_0}\e_1^{1+bk_1}\cdots\e_{r+1}^{1+bk_{r+1}}} \dd t_0\\
    &\ll\sum_{\e_1, \dots, \e_{r+1}} \frac{\theta'(\e_1, \dots,
      \e_{r+1})B^{1+b}T^{-bk_0}}{\e_1^{1+bk_1}\cdots\e_{r+1}^{1+bk_{r+1}}}\\
    &\ll\sum_{\e_1, \dots, \e_r} \frac{\M(\theta'(\e_1, \dots, \e_{r+1}),\e_{r+1})B}{\e_1\cdots\e_r}\\
    &\ll B(\log B)^r.
  \end{split}\]
  
  If $k_0> 0$, the computations are similar.

  If $k_0=0$, we split the summation over $\e_1, \dots, \e_{r+1}$ into two
  subsets, the first defined by $\e_1^{k_1}\cdots\e_{r+1}^{k_{r+1}}\le
  B$.

  Here, we compute
  \[\begin{split}
    \ll{}& \sum_{\e_1, \dots, \e_{r+1}} \theta'(\e_1,
    \dots, \e_{r+1}) \int_1^{T}
    \frac{B^{1-a}}{t_0\e_1^{1-ak_1}\cdots\e_{r+1}^{1-ak_{r+1}}} \dd t_0\\
    \ll{}& \sum_{\e_1, \dots, \e_{r+1}} \frac{\theta'(\e_1, \dots,
      \e_{r+1})B^{1-a}(\log \log
      B)}{\e_1^{1-ak_1}\cdots\e_{r+1}^{1-ak_{r+1}}}\\
    \ll{}& B(\log B)^r(\log \log B).
  \end{split}\] For the subset defined by $\e_1^{k_1}\cdots\e_{r+1}^{k_{r+1}}>
  B$, the computation is similar.
\end{proof}

\section{Completion of summations}\label{sec:generic_completion}

Let $r,s \in \ZZnn$ with $r \ge s$. In this section, we consider functions
\[\theta_{r+s}: \ZZnn^{r+s} \to \RR, \qquad V_{r+s}: \RRnn^{r+s} \times \RR_{\ge
  3} \to \RR.\] In the previous section, we summed the product of such
functions over one variable; here, we sum over all variables and therefore
want to estimate
\begin{equation*}
  \sum_{\e_1, \dots, \e_{r+s}} \theta_{r+s}(\e_1, \dots,
  \e_{r+s})V_{r+s}(\e_1, \dots, \e_{r+s};B).
\end{equation*}
This will be done in the case that
$\theta_{r+s}$ and $V_{r+s}$ fulfill certain conditions described in the
following that allow us to apply Proposition~\ref{prop:complete_summation_1}
repeatedly.

For the implied constants in this section, we use a similar convention as
described in Remark~\ref{rem:implied_constants}, i.e., the implied constants
are meant to be independent of $\e_1, \dots, \e_{r+s}$ and $B$, but may depend
on everything else, in particular on $V_{r+s}$ and $\theta_{r+s}$.

For $V_{r+s}: \RRnn^{r+s} \times \RR_{\ge 3} \to \RR$ a non-negative function,
we require the following, similar to Section~\ref{sec:generic_errors}. We
assume that, for $j = 1, \dots, s$, we have $a_j \in \RRp$ and
\begin{align*}
  k_{1,j}, \dots, k_{r-s+j-1,j} &\in \RR, & k_{r-s+j,j} &\in \RRnz, &
  k_{r-s+j+1,j}, \dots, k_{r,j} &= 0,\\
  k_{r+1,j}, \dots, k_{r+j-1,j} &\in \RR,& k_{r+j,j} &\in \RRnz,&
  k_{r+j+1,j},\dots, k_{r+s,j}&=0.
\end{align*}

For $\ell=1, \dots, s$ and $i = 1, \dots, r+s$, we define
\begin{equation*}
  A^{(\ell)}=\sum_{j=1}^\ell a_j,\quad A_i^{(\ell)}= \sum_{j=1}^\ell a_jk_{i,j}.
\end{equation*}
We assume that
\begin{equation}\label{eq:completion_bound}
  V_{r+s}(\e_1, \dots, \e_{r+s};B) \ll \frac{B^{1-A^{(s)}}}{\e_1^{1-A_1^{(s)}}\cdots
    \e_{r+s}^{1-A_{r+s}^{(s)}}},
\end{equation}
and that $V_{r+s}(\e_1, \dots, \e_{r+s};B) = 0$
unless both
\begin{equation}\label{eq:bound_s_completion}
  \e_1^{k_1,j}\cdots\e_{r+s}^{k_{r+s,j}} =
  \e_1^{k_1,j}\cdots\e_{r+j}^{k_{r+j,j}} \le B,
\end{equation}
for $j=1, \dots, s$, and
\begin{equation}\label{eq:bound_r_completion}
  1 \le \e_i \le B,
\end{equation}
for $i=1, \dots, r+s$.

For $\ell = r+s-1, \dots, 0$, we define recursively
\begin{equation}\label{eq:V_integral}
  \begin{split}
    V_\ell(\e_1, \dots, \e_\ell;B) &= \int_{\e_{\ell+1}} V_{\ell+1}(\e_1,
    \dots, \e_{\ell+1};B)
    \dd \e_{\ell+1}\\
    &= \int_{\e_{\ell+1}, \dots, \e_{r+s}}V_{r+s}(\e_1, \dots, \e_{r+s}) \dd
    \e_{r+s} \cdots \dd \e_{\ell+1}
  \end{split}
\end{equation}
and assume that $V_\ell$ as a function $\e_{\ell}$ has a continuous derivative
whose sign changes only finitely often.

\begin{lemma}\label{lem:volume_bound}
  In the situation described above, we have, for $\ell \in \{1, \dots,
  s\}$, \[V_{r+\ell}(\e_1, \dots, \e_{r+\ell};B) \ll
  \frac{B^{1-A^{(\ell)}}}
  {\e_1^{1-A_{1}^{(\ell)}}\cdots\e_{r+\ell}^{1-A_{r+\ell}^{(\ell)}}}\]
  and, for $\ell \in \{1, \dots, r\}$, \[V_\ell(\e_1, \dots, \e_\ell;B) \ll
  \frac{B(\log B)^{r-\ell}}{\e_1\cdots\e_\ell}.\]
\end{lemma}

\begin{proof}
  The proof is analogous to the proof of
  Lemma~\ref{lem:summation_V_theta_bounds}, skipping the step of replacing
  sums by integrals via Lemma~\ref{lem:sum_over_n}.
\end{proof}

Recall the notation of Definition~\ref{def:arithmetic_function_with_average}
and Definition~\ref{def:arithmetic_function_r_with_average}.

\begin{definition}\label{def:arithmetic_function_r_summable}
  Let $C \in \RRnn$. Let $\Theta_{2,0}(C)$ be the set $\RR$ of real
  numbers. For any $r \in \ZZp$, we define $\Theta_{2,r}(C)$ recursively as
  the set of all functions $\theta: \ZZp^r \to \RR$ in the variables $\e_1,
  \dots, \e_r$ such that $\theta \in \Theta_{1,r}(C,\e_r)$ and $\theta'
  \in \Theta_{2,r-1}(C)$, where $\theta'(\e_1, \dots, \e_{r-1}) =
  \A(\theta(\e_1, \dots, \e_r),\e_r)$.

  For $\theta \in \Theta_{2,r}(C)$ and any pairwise distinct $i_1, \dots, i_n
  \in \{1, \dots, r\}$, we define
  \begin{equation*}
    \A(\theta(\e_1, \dots, \e_r),
    \e_{i_1}, \dots, \e_{i_n}) = \A(\dots\A(\theta(\e_1, \dots,
    \e_r),\e_{i_1})\dots,\e_{i_n});
  \end{equation*}
  it is a function in $\Theta_{2,r-n}(C)$.
\end{definition}

\begin{prop}\label{prop:several_summations}
  Let $V_{r+s}$ be as described before Lemma~\ref{lem:volume_bound}, and let
  $\theta_{r+s} \in \Theta_{2,r+s}(C)$ for some $C \in \RRnn$. Then
  \begin{multline*}
      \sum_{\e_1, \dots, \e_{r+s}} \theta_{r+s}(\e_1, \dots,
      \e_{r+s})V_{r+s}(\e_1, \dots, \e_{r+s};B)\\
      ={} c_0
      \int_{\e_1, \dots, \e_{r+s}} V_{r+s}(\e_1, \dots, \e_{r+s};B) \dd
      \e_{r+s}\cdots \dd \e_1\\
      + O\left(B(\log B)^{r-1}(\log \log B)^{\max\{1,s\}}\right),
  \end{multline*}
  where $c_0=\A(\theta_{r+s}(\e_1, \dots, \e_{r+s}),\e_{r+s}, \dots, \e_1)$.
\end{prop}

\begin{proof}
  We proceed by induction as follows, for $\ell=r+s, \dots, 1$. Given
  $\theta_\ell \in \Theta_{2,\ell}(C)$, we define $\theta_{\ell-1} \in
  \Theta_{2,\ell-1}(C)$ by
  \begin{equation*}
    \begin{split}
      \theta_{\ell-1}(\e_1, \dots, \e_{\ell-1}) &= \A(\theta_{\ell}(\e_1,
      \dots,
      \e_{\ell}), \e_{\ell})\\
      &= \A(\theta_{r+s}(\e_1, \dots, \e_{r+s}),\e_{r+s}, \dots, \e_{\ell}).
    \end{split}
  \end{equation*}

\begin{table}[ht]
  \centering
  \[\begin{array}[h]{|c||c|c|}
    \hline
    \text{Proposition~\ref{prop:complete_summation_1}} & \ell \in
    \{1, \dots, r\} & \ell \in \{r+1, \dots, r+s\} \\
    \hline\hline
    (r,s) & (\ell-1,0) & (r-1,\ell-r) \\
    \hline
    \e_0 & \e_\ell & \e_\ell\\
    \e_1, \dots, \e_r & \e_1, \dots, \e_{\ell-1} & \e_1, \dots,
    \e_{\ell-s-1}, \e_{\ell-s+1}, \dots, \e_r\\
    \e_{r+s}, \dots, \e_{r+s} & - & \e_{r+1}, \dots, \e_{\ell-1},
    \e_{\ell-s}\\
    \hline
    \theta \in \Theta_{1,r+s+1}(C) & \theta_\ell \in \Theta_{2,\ell}(C) &
    \theta_\ell \in \Theta_{2,\ell}(C)\\
    \A(\theta(\e_0, \dots, \e_{r+s}),\e_0) & \theta_{\ell-1} \in
    \Theta_{2,\ell-1}(C) & \theta_{\ell-1} \in \Theta_{2,\ell-1}(C) \\
    \hline
    V & V_\ell/(\log B)^{r-\ell} & V_\ell\\
    V' & V_{\ell-1}/(\log B)^{r-\ell} & V_{\ell-1}\\
    \hline
    k_{0,j}, k_{1,j}, \dots, k_{r+s,j} & - & k_{1,j}, \dots, k_{\ell,j}\\
    & & \text{arranged as $\e_1, \dots, \e_\ell$,}\\
    \hline
    A; A_0, A_1, \dots, A_{r+s} & - & \text{$A^{(\ell-r)}$;
      $A_{1}^{(\ell-r)}, \dots, A_{\ell}^{(\ell-r)}$} \\
    & & \text{arranged as $\e_1,
      \dots, \e_\ell$,}\\
    \hline
    (\ref{eq:generic_Vbound}) & \text{Lemma~\ref{lem:volume_bound}} &
    \text{Lemma~\ref{lem:volume_bound}}\\
    (\ref{eq:height_s}) & - &
    (\ref{eq:bound_s_completion})\\
    (\ref{eq:height_r}) & (\ref{eq:bound_r_completion}) &
    (\ref{eq:bound_r_completion})\\
    \hline
  \end{array}\]
  \caption{Application of Proposition~\ref{prop:complete_summation_1}.}
  \label{tab:complete_summation}
\end{table}

  With $V_\ell, V_{\ell-1}$ as in (\ref{eq:V_integral}), we apply
  Proposition~\ref{prop:complete_summation_1} to show that
  \begin{multline*}
    \sum_{\e_\ell} \theta_{\ell}(\e_1, \dots, \e_\ell)V_\ell(\e_1, \dots,
    \e_\ell;B)\\ = \theta_{\ell-1}(\e_1, \dots, \e_{\ell-1})V_{\ell-1}(\e_1,
    \dots, \e_{\ell-1};B) + R(\e_1, \dots, \e_{\ell-1};B),
  \end{multline*}
  where
  \begin{equation*}
    \sum_{\e_1, \dots, \e_{\ell-1}} R(\e_1, \dots, \e_{\ell-1};B) \ll B(\log
      B)^{r-1}(\log \log B)^{\max\{1,\ell-r\}}.
  \end{equation*}
  How to apply Proposition~\ref{prop:complete_summation_1} (especially with
  respect to the order of the variables $\e_1, \dots, \e_{\ell}$) depends on
  whether $1 \le \ell \le r$ or $r+1 \le \ell \le r+s$; furthermore, there are
  many prerequisites to check. Therefore, we have listed the details for the
  application of Proposition~\ref{prop:complete_summation_1} in
  Table~\ref{tab:complete_summation}.
\end{proof}

\begin{remark}
  An analogous result to Proposition~\ref{prop:several_summations} holds if we
  want to estimate $\theta_{r+1}(\e_1, \dots, \e_{r+1})V_{r+1}(\e_1, \dots,
  \e_{r+1};B)$ summed over $\e_1, \dots, \e_{r+1}$, but with
  (\ref{eq:completion_bound}) and (\ref{eq:bound_s_completion}) replaced by a
  bound analogous to (\ref{eq:generic_Vbound_double}). In the proof, we apply
  Proposition~\ref{prop:complete_summation_2} instead of
  Proposition~\ref{prop:complete_summation_1} in the first summation over
  $\e_{r+1}$.
\end{remark}

\section{Real-valued functions}

The following result is often useful to derive bounds such as
(\ref{eq:generic_Vbound}), (\ref{eq:generic_Vbound_double}) and
(\ref{eq:completion_bound}) for real-valued functions defined through certain
integrals; for example, we recover the bounds of
\cite[Lemma~8]{arXiv:0710.1560}.

\begin{lemma}\label{lem:generic_bounds}
  Let $a,b \in \RRnz$. Then we have the following bounds.
  \begin{enumerate}[(1)]
  \item\label{it:generic_bounds_a} $\int_{|at^2+b|\le 1} \dd t \ll
    \min\{|a|^{-1/2},|ab|^{-1/2}\}$.
  \item\label{it:generic_bounds_a1} $\int_{|at^2u+bu^k| \le 1} \dd t \dd u \ll
    |ab^{1/k}|^{-1/2}$.
  \item\label{it:generic_bounds_a2} $\int_{|at^2+bu^k| \le 1} \dd t \dd u \ll
    |a|^{-1/2}|b|^{-1/k}$, for $k > 2$.
  \item\label{it:generic_bounds_b} $\int_{|at^2+bt| \le 1} \dd t \ll
    \min\{|a|^{-1/2},|b|^{-1}\}$.
  \item\label{it:generic_bounds_b1} $\int_{|at^2u+btu^2| \le 1} \dd t \dd u \ll
    |ab|^{-1/3}$.
  \item\label{it:generic_bounds_b2} $\int_{|at^2+btu^k| \le 1} \dd t \dd u \ll
    |a|^{-(k-1)/(2k)}|b|^{-1/k}$, for $k > 1$.
  \end{enumerate}
\end{lemma}

\begin{proof}
  We treat only the case $a>0$; its opposite is essentially the same.

  For \eqref{it:generic_bounds_a}, we consider $t$ such that $|at^2+b| \le 1$;
  if there is no such $t$, the claim is obvious. Otherwise, suppose first $|b|
  \le 2$. Then $|at^2+b| \le 1$ implies $|at^2| \le 3$, i.e., $t \ll
  |a|^{-1/2} \ll |ab|^{-1/2}$. Next, suppose $|b| > 2$. Obviously $b>2$ is
  impossible, so we assume $b<-2$. Then $|at^2+b| \le 1$ implies
  \[\sqrt{\frac{-b-1}{a}} \le t \le \sqrt{\frac{-b+1}{a}}.\] We note that the
  condition $\sqrt{x} \le t \le \sqrt{x+y}$ for $x,y>0$ describes an interval
  of length $\ll x^{-1/2}y$. Here $x=(b-1)/a>b/(2a)$ and $y=2/a$, so the
  interval for $t$ has length $\ll |ab|^{-1/2} \ll |a|^{-1/2}$.

  For \eqref{it:generic_bounds_a1}, we apply \eqref{it:generic_bounds_a} and
  obtain 
  \begin{multline*}
    \int_{|at^2u+bu^2| \le 1} \dd t \dd u \ll \int_0^\infty
    \min\{|au|^{-1/2},|abu^{k+1}|^{-1/2}\} \dd u\\
    \ll \int_0^{|b|^{-1/k}} |au|^{-1/2} \dd u + \int_{|b|^{-1/k}}^\infty
    |abu^{k+1}|^{-1/2} \dd u \ll \frac{1}{|ab^{1/k}|^{1/2}}.
  \end{multline*}

  Similarly, for \eqref{it:generic_bounds_a2}, we get 
  \begin{multline*}
    \int_{|at^2+bu^k| \le 1}\dd t \dd u \ll \int_0^\infty
    \min\{|au|^{-1/2},|abu^3|^{-1/2}\} \dd u\\
    \ll \int_0^{|b|^{-1/k}} |a|^{-1/2} \dd u + \int_{|b|^{-1/k}}^\infty
    |abu^k|^{-1/2} \dd u \ll \frac{1}{|a|^{1/2}|b|^{1/k}}.
  \end{multline*}

  For \eqref{it:generic_bounds_b}, we transform $|at^2+bt| \le 1$ to
  \[\sqrt{\max\left\{0,\frac{b^2-4a}{4a^2}\right\}} \le |t+b/(2a)| \le
  \sqrt{\frac{b^2+4a}{4a^2}}.\] If $b^2 \le 8a$ then $((b^2+4a)/(4a^2))^{1/2}
  \ll |a|^{-1/2} \ll |b|^{-1}$, which is also a bound for the length of the
  interval of allowed values of $t$. If $b^2 > 8a$, then we apply the above
  bound for $x = (b^2-4a)/(4a^2)>b^2/(8a^2)$ and $y=2/a$ to conclude that the
  interval for $t$ has length $\ll |b|^{-1} \ll |a|^{-1/2}$.

  For \eqref{it:generic_bounds_b1}, we apply \eqref{it:generic_bounds_b} to
  conclude 
  \begin{multline*}
    \int_{|at^2u+btu^2| \le 1} \dd t \dd u \ll
    \int_0^\infty \min\{|au|^{-1/2},|bu^2|^{-1}\} \dd u\\
    \ll \int_0^{|a/b^2|^{1/3}} |au|^{-1/2} \dd u +
    \int_{|a/b^2|^{1/3}}^\infty |bu^2|^{-1} \dd u \ll \frac{1}{|ab|^{1/3}}.
  \end{multline*}

  For \eqref{it:generic_bounds_b2}, we have 
  \begin{multline*}
    \int_{|at^2+btu^k| \le 1} \dd t \dd u \ll \int_0^\infty
    \min\{|a|^{-1/2},|bu^k|^{-1}\} \dd u\\
    \ll \int_0^{|a^{1/2}/b|^{1/k}} |a|^{-1/2} \dd
    u+\int_{|a^{1/2}/b|^{1/k}}^\infty |bu^k|^{-1} \dd u \ll
    \frac{1}{|a|^{(k-1)/(2k)}|b|^{1/k}}.
  \end{multline*}
  This completes the proof.
\end{proof}

\section{Arithmetic functions in one variable}

In Section~\ref{sec:generic_errors} and Section~\ref{sec:generic_completion},
we were interested in the average size of arithmetic functions on intervals,
with certain bounds on the error term.

In this section, we describe a set of functions in one variable
(Definition~\ref{def:function_small}) for which this information is computable
explicitly (by Corollary~\ref{cor:function_small_has_average}). This includes
the functions $f_{a,b}$ treated in \cite[Lemma~1]{arXiv:0710.1560} (see
Example~\ref{exa:f_ab}).

\begin{lemma}\label{lem:sum_arithmetic_general}
  Let $\theta: \ZZp \to \RR$ be a function, and let $t,y \in \RRnn$, with $y
  \le t$. Let $a,q \in \ZZp$, with $\cp a q$.  If the infinite sum \[\sums{d
    > 0\\\cp d q} \frac{(\theta*\mu)(d)}{d}\] converges to $c_0 \in \RR$, we
  have
  \begin{multline*}
    \sum_{\substack{0 < n \le t\\\congr n a q}} \theta(n) =
    \frac{c_0t}{q} + O\left(\sum_{\substack{0 < d \le y\\\cp d q}}
      |(\theta*\mu)(d)|\right.\\\left.  + \frac t q \cdot
      \Bigg|\sum_{\substack{d > y\\\cp d q}} \frac{(\theta*\mu)(d)}{d}\Bigg| +
      \sum_{\substack{0 < n < t/y\\\cp n q}} \Bigg|\sum_{\substack{y < d <
          t/n\\\congr{nd}a q}} (\theta*\mu)(d)\Bigg|\right).
  \end{multline*}
\end{lemma}

\begin{proof}
  Since $\theta = (\theta*\mu)*1$, we have
  \begin{equation*}
    \sum_{\substack{0 < n \le t\\\congr n a q}} \theta(n) =
    \sum_{\substack{0 < n \le t\\\congr n a q}} \sum_{d \div n}
    (\theta*\mu)(d)= \sum_{\substack{0 < d \le t\\\cp d q}}
    \sum_{\substack{0 < n' \le t/d\\\congr{n'd}a q}}
    (\theta*\mu)(d).
  \end{equation*}
  Splitting this sum into the cases $d \le y$ and its opposite, we get
  \begin{equation*}
    = \sum_{\substack{0 < d \le y\\\cp d q}} (\theta*\mu)(d) \cdot
    \left(\frac t{qd}+O(1)\right) + \sums{0 < n' \le
        t/y\\\cp {n'}q} 
    \sums{y < d < t/n'\\\congr{n'd}a q} (\theta*\mu)(d),
  \end{equation*}
  and the result follows.
\end{proof}

\begin{lemma}\label{lem:sum_arithmetic}
  Let $C \in \RR_{\ge 1}$. Let $\theta: \ZZp \to \RR$ be such that, for any $t
  \in \RRnn$,
  \begin{equation*}
    \sum_{0 < n \le t} |(\theta*\mu)(n)|\cdot n \le t(\log(t+2))^{C-1}.
  \end{equation*}
  Then, for any $q \in \ZZp$ and $a \in \ZZ$ with $\cp a q$, the real number
  $c_0$ as in Lemma~\ref{lem:sum_arithmetic_general} exists, and
  \begin{equation*}
    \sum_{\substack{0 < n \le t\\\congr{n}{a}{q}}} \theta(n)=
    \frac{c_0t}{q} + O_C\left((\log(t+2))^{C}\right).
  \end{equation*}
\end{lemma}

\begin{proof}
  We apply Lemma~\ref{lem:sum_arithmetic_general}, with $y=t$. It remains to
  handle the error term, whose third part clearly vanishes.  By
  Lemma~\ref{lem:sum_over_n} and our assumption on $\theta$, the first part of
  the error term is
  \begin{equation*}
    \sum_{0 < n \le t} |(\theta*\mu)(n)| \ll_C
    (\log(t+2))^{C},
  \end{equation*}
  and the second part of the error term is
  \begin{equation*}
    \frac{t}{q}\sum_{n > t} \frac{|(\theta*\mu)(n)|}{n} \ll_C
    q^{-1}(\log(t+2))^{C-1}.
  \end{equation*}
  This completes the proof.
\end{proof}

\begin{remark}
  For infinite products, we use the following convention.  We require that the
  partial products of all non-vanishing factors of an infinite product
  converge to a non-zero number. If there are any vanishing factors, the value
  of the infinite product is zero. Otherwise, the infinite product cannot
  converge to zero.
\end{remark}

Let $\P$ denote the set of all primes.

\begin{definition}\label{def:function_product}
  Let $\Theta_1$ be the set of all non-negative functions $\theta: \ZZp \to
  \RR$ such that there is a $c\in \RR$ and a system of non-negative functions
  $A_p: \ZZnn \to \RR$ for $p \in \P$ satisfying
  \begin{equation*}
    \theta(n) = c \prodpdd{\nu}{n}{A_p(\nu)}\prodpn{n}{A_p(0)}
  \end{equation*}
  for all $n \in \ZZ$ (where the first product is over all $p \in \P$ and $\nu
  \in \ZZp$ such that $p^\nu \div n$ but $p^{\nu+1} \ndiv n$). In this
  situation, we say that $\theta \in \Theta_1$ \emph{corresponds to} $c,A_p$.
\end{definition}

\begin{lemma}\label{lem:theta_product_uniqueness}
  Suppose $\theta \in \Theta_1$ is not identically zero and corresponds to $c,
  A_p$ and $c',A'_p$. Then there are unique $b_p \in \RRp$, for $p \in \P$,
  such that $\prodp{b_p}$ converges to a number $b_0 \in \RRp$, $A'_p(\nu) =
  b_p A_p(\nu)$ for all $p \in \P$, $\nu \in \ZZnn$, and $c' = c/b_0$.

  Conversely, given $\theta \in \Theta_1$ corresponding to $c, A_p$, and
  $b_p \in \RRp$, for $p \in \P$, such that $b_0 = \prodp{b_p} \in \RRp$
  exists. Then $\theta$ also corresponds to $c', A'_p$ defined as $c' =
  c/b_0$ and $A'_p(\nu) = b_pA_p(\nu)$ for all $p \in \P$, $\nu \ge 0$.
\end{lemma}

\begin{proof}
  Fix $n = \prodp{p^{k(p)}} \in \ZZp$ such that $\theta(n) \ne 0$. Then
  $A_p(k(p))$ and $A'_p(k(p))$ are non-zero, so $b_p \in \RRp$ is
  uniquely defined as
  $A'_p(k(p))/A_p(k(p))$. Since \[\frac{A_p(\nu)}{A_p(k(p))} =
  \frac{\theta(p^{\nu-k(p)}n)}{\theta(n)} = \frac{A'_p(\nu)}{A'_p(k(p))},\] we
  have $A'_p(\nu) = b_p A_p(\nu)$ for all $\nu \in \ZZnn$.

  Since $\prodpn{n}{A_p(0)}$ and $\prodpn{n}{A'_p(0)}$ are well-defined
  non-zero numbers, also $\prodpn{n}{b_p} \in \RRp$ and therefore $b_0 \in
  \RRp$ exist. Since
  \begin{equation*}
    \theta(n) = c' \prodpdd{\nu}{n}{A'_p(\nu)}\prodpn{n}{A'_p(0)} = c'
    b_0 \prodpdd{\nu}{n}{A_p(\nu)}\prodpn{n}{A_p(0)},
  \end{equation*}
  we conclude that $c=c'b_0$.

  It is straightforward to check the converse statement.
\end{proof}

\begin{definition}\label{def:function_small}
  For any $b \in \ZZp$, $C_1,C_2,C_3\in \RR_{\ge 1}$, let
  $\Theta_2(b,C_1,C_2,C_3)$ be the set of all functions $\theta \in \Theta_1$
  for which there exist corresponding $c, A_p$ satisfying the following
  conditions.
  \begin{enumerate}[(1)]
  \item\label{it:function_small_1} For all $p \in \P$ and $\nu \ge 1$,
    \[|A_p(\nu)-A_p(\nu-1)| \le \begin{cases}
      C_1, &p^\nu \div b,\\
      C_2p^{-\nu}, &p^\nu \ndiv b;
    \end{cases}\] 
  \item\label{it:function_small_2} For all $k \in \ZZp$, we have
    $\left|c\prodpn{k}{A_p(0)}\right| \le C_3$.
  \end{enumerate}

  Given $\theta \in \Theta_2(b,C_1,C_2,C_3)$, we will see in
  Proposition~\ref{prop:theta_mu_involution} that, for any $q \in \ZZp$, the
  infinite product
  \begin{equation*}
    c\prodpnp{q}{\left(1-\frac 1 p\right) \sum_{\nu=0}^\infty
      \frac{A_p(\nu)}{p^\nu}} \prodpd{q} A_p(0)
  \end{equation*}
  converges to a real number, which we denote as $\A(\theta(n),n,q)$.
\end{definition}

If $A_p(\nu)=A_p(\nu+1)$ for all primes $p$ and all $\nu \ge 1$, then the
formula is simplified to
\begin{equation*}
  A(\theta(n),n,q) = c \prodpnp{q}{\left(1-\frac 1 p\right)A_p(0)+\frac 1
    p A_p(1)} \prodpd{q}{A_p(0)}.
\end{equation*}

We will see in Corollary~\ref{cor:function_small_has_average} how the notation
$\A(\theta(n),n,q)$ of Definition~\ref{def:function_small} is related to the
notation $\A(\theta(n),n)$ of
Definition~\ref{def:arithmetic_function_with_average}.

\begin{remark}
  If $\theta \in \Theta_2(b,C_1,C_2,C_3)$ corresponds to $c,A_p$ and
  $c',A_p'$, where $c,A_p$ satisfy conditions~\eqref{it:function_small_1},
  \eqref{it:function_small_2} of Definition~\ref{def:function_small}, then
  $c',A_p'$ do not necessarily satisfy these conditions. However, with $b_p
  \in \RRp$ as in Lemma~\ref{lem:theta_product_uniqueness}, if we replace
  $C_1,C_2,C_3$ by \[C_1\max_{p\div b}\{b_p\},\qquad C_2\max_p\{b_p\},\qquad
  C_3\prods{p\\|b_p|>1} b_p,\] then $c',A_p'$ satisfy
  conditions~\eqref{it:function_small_1}, \eqref{it:function_small_2}.
\end{remark}

In all statements regarding $\theta \in \Theta_2(b,C_1,C_2,C_3)$, we will mark
explicitly by subscripts if an implied constant in the notation $\ll$ and
$O(\dots)$ depends on any of $b, C_1, C_2, C_3$ or $\theta$. The reason is that
we will apply the results of this section in the following
Section~\ref{sec:arithmetic_r} to functions in several variables $\e_1, \dots,
\e_r$. As functions in $\e_r$, they will lie in $\Theta_2(b,C_1,C_2,C_3)$, but
(some of) $b,C_1,C_2,C_3$ will depend on $\e_1, \dots, \e_{r-1}$.

\begin{prop}\label{prop:theta_mu_involution}
  Let $\theta \in \Theta_1$ be non-trivial, with corresponding $c,A_p$.
  \begin{enumerate}[(1)]
  \item\label{it:theta_mu_involution_single} For any $n \in \ZZp$,
    \begin{equation*}
      (\theta * \mu)(n) = c \prodpn{n}{A_p(0)}
      \prodpddp{\nu}{n}{A_p(\nu)-A_p(\nu-1)}.
    \end{equation*}
  \item\label{it:theta_mu_involution_sum_bound} We assume $\theta \in
    \Theta_2(b,C_1,C_2,C_3)$. For any $t \in \RRnn$,
    \begin{equation*}
      \sum_{0 < n \le t} |(\theta*\mu)(n)|\cdot n \ll_{C_2}
      \tau(b)(C_1C_2)^{\omega(b)}C_3 t(\log(t+2))^{C_2-1},
    \end{equation*}
    where $\tau(n)=\sum_{d|n} 1$ is the divisor function.
  \item\label{it:theta_mu_involution_sum} We assume $\theta \in
    \Theta_2(b,C_1,C_2,C_3)$. For any $q \in \ZZp$, the infinite sum
    and the infinite product
    \begin{equation*}
      \sum_{\substack{n > 0\\\cp n q}} \frac{(\theta*\mu)(n)}{n}, \qquad 
      c\prodpnp{q}{\left(1-\frac 1 p\right) \sum_{\nu=0}^\infty
        \frac{A_p(\nu)}{p^\nu}} \prodpd{q} A_p(0).
    \end{equation*}
    converge to the same real number.
  \end{enumerate}
\end{prop}

\begin{proof}
  Up to the converging product $\prodpn{n}A_p(0)$, claim
  \eqref{it:theta_mu_involution_single} is an identity of finite algebraic
  expressions:
  \begin{equation*}
    \begin{split}
      &c \prodpn{n}{A_p(0)}
      \prodpddp{\nu}{n}{A_p(\nu)-A_p(\nu-1)}\\
      ={}&\sums{d \div n\\|\mu(d)=1|} c \prodpn{n}{A_p(0)}
      \prodpddn{\nu}{n}{d} A_p(\nu) \prodpddd{\nu}{n}{d}(-A_p(\nu-1))\\
      ={}&\sum_{d \div n} \mu(d) c\prodpn{\frac n
        d}{A_p(0)}\prodpdd{\nu}{\frac n d}{A_p(\nu)}\\
      ={}&\sum_{d \div n} \mu(d)\theta(n/d)\\
      {}=& (\theta*\mu)(n).
    \end{split}
  \end{equation*}

  For \eqref{it:theta_mu_involution_sum_bound}, it follows from
  \eqref{it:theta_mu_involution_single} that
  \begin{equation*} |(\theta*\mu)(n)|
    \le C_1^{\omega(\gcd(b,n))}C_2^{\omega(n)}C_3\gcd(b,n)n^{-1}.
  \end{equation*}
  Therefore,
  \begin{equation*}
    \begin{split}
      \sum_{0 < n \le t} |(\theta*\mu)(n)|\cdot n &\ll \sum_{0 < n \le t}
      C_1^{\omega(\gcd(n,b))}C_2^{\omega(n)}C_3\gcd(n,b)\\
      &\ll \sum_{d|b} \sums{0 < n' \le t/d\\\cp {n'}b}
      C_1^{\omega(d)}C_2^{\omega(dn')}C_3d\\
      &\ll_{C_2} \sum_{d\div b}(C_1C_2)^{\omega(d)}C_3t(\log(t+2))^{C_2-1}\\
      &\ll \tau(b)(C_1C_2)^{\omega(b)}C_3t(\log(t+2))^{C_2-1},
    \end{split}
  \end{equation*}
  using Example~\ref{exa:phis_phid_omega}.

  For \eqref{it:theta_mu_involution_sum}, for $p \in \P$, let $\nu_p =
  \min\{\nu \in \ZZnn \where A_p(\nu) \ne 0\}$. Since
  $\theta$ is non-trivial, $\nu_p = 0$ for all but finitely many $p$, so $a =
  \prodp{p^{\nu_p}}$ defines a positive integer. If $a \ndiv n$, then
  $\theta(n) = 0$ and $(\theta*\mu)(n)=0$.

  We define the multiplicative function $B: \ZZp \to \RR$ by
  \begin{equation*}
    B(p^\nu) =
    \frac{A_p(\nu+\nu_p)-A_p(\nu+\nu_p-1)}{A_p(\nu_p)},
  \end{equation*}
  for any $p \in \P$ and $\nu \in \ZZp$, and
  \begin{equation*}
    c' = c  \prodp{A_p(\nu_p)} \in \RR.
  \end{equation*}
  If $n=an'$ for some $n' \in \ZZp$, then, by
  \eqref{it:theta_mu_involution_single},
  \begin{equation*}
       (\theta * \mu)(n) = c \prodpn{an'}{A_p(0)}
      \prodpddp{\nu}{an'}{A_p(\nu)-A_p(\nu-1)}=c' B(n').
  \end{equation*}
  
  We assume that $\cp a q$.  By~\eqref{it:theta_mu_involution_sum_bound} and
  Lemma~\ref{lem:sum_over_n}, the following sum converges absolutely, so that
  we may form the Euler product in the second step.
  \begin{equation*}
    \begin{split}
      &\sums{n=1\\\cp n q}^\infty \frac{(\theta*\mu)(n)}{n} = \sums{n'=1\\\cp
        {n'} q}^\infty \frac{c'B(n')}{an'}= \frac{c'}{a}
      \prodpnp{q}{\sum_{\nu=0}^\infty \frac{B(p^\nu)}{p^\nu}}\\
      ={}& c \prodp{\frac{A_p(\nu_p)}{p^{\nu_p}}} \prodpnp{q}
      {1+\sum_{\nu=1}^\infty \frac{A_p(\nu+\nu_p)-A_p(\nu+\nu_p-1)}{p^\nu
          A_p(\nu_p)}}\\
      ={}& c \prodpd{q}{\frac{A_p(\nu_p)}{p^{\nu_p}}} \prodpnp{q}{\left(1-\frac
          1 p\right) \sum_{\nu=\nu_p}^\infty \frac{A_p(\nu)}{p^\nu}}.
    \end{split}
  \end{equation*}
  Since $A_p(\nu) = 0$ for any $\nu < \nu_p$, and $\nu_p = 0$ for any $p \div
  q$, this proves the claim in the case $\cp a q$.

  If $\ncp a q$, then $(\theta*\mu)(n)=0$ for all $n$ satisfying $\cp n q$, so
  that \eqref{it:theta_mu_involution_sum} is trivially true.
\end{proof}

Because of the following result, $\A(\theta(n),n,q)$ should be
viewed as the average size of $\theta(n)$ when summed over all $n$ in a
residue class modulo $q$ in a sufficiently long interval.

\begin{cor}\label{cor:function_small_has_average}
  Let $\theta \in \Theta_2(b,C_1,C_2,C_3)$ be non-trivial. If $q \in \ZZp$
  and $a \in \ZZ$ with $\cp a q$, then
  \begin{equation*}
    \sum_{\substack{0 < n \le t\\\congr{n}{a}{q}}} \theta(n)=
    \frac{t}{q} \A(\theta(n),n,q)  +
    O_{C_2}\left(\tau(b)(C_1C_2)^{\omega(b)}C_3(\log(t+2))^{C_2}\right).
  \end{equation*}
  for any $t \in \RRnn$. In particular, in the notation of
  Definition~\ref{def:arithmetic_function_with_average}, $\theta \in
  \Theta_0(C_2)$, with $\A(\theta(n),n) = \A(\theta(n),n,1)$ and
  $\E(\theta(n),n) = O_{C_2}(\tau(b)(C_1C_2)^{\omega(b)}C_3)$.
\end{cor}

\begin{proof}
  Let $C_4=\tau(b)(C_1C_2)^{\omega(b)}C_3$. By
  Proposition~\ref{prop:theta_mu_involution}\eqref{it:theta_mu_involution_sum_bound},
  Lemma~\ref{lem:sum_arithmetic} applies to $C_4^{-1}\theta$, with $c_0 =
  C_4^{-1}\A(\theta(n),n,q)$ by
  Proposition~\ref{prop:theta_mu_involution}\eqref{it:theta_mu_involution_sum}.
\end{proof}

\begin{example}\label{exa:f_ab}
  For $a, b\in \ZZp$, we consider $f_{a,b}$ as in
  \cite[(3.2)]{arXiv:0710.1560}. Then $f_{a,b} \in \Theta_1$, corresponding to
  $c,A_p$, where $c=1$ and $A_p(0) = 1$ for any prime $p$, while
  \begin{equation*}
    A_p(\nu) =
    \begin{cases}
      0, &p \div b,\\
      1, &p \ndiv b,\ p \div a,\\
      1-\frac 1 p, &p\ndiv ab.\\
    \end{cases}
  \end{equation*}
  for any $\nu > 0$. Clearly $f_{a,b} \in \Theta_2(\prodpd{b}p,1,1,1)$, and we
  compute
  \[\A(f_{a,b}(n),n,q) = \prodpdnp{b}{q}{1-\frac 1 p} \prodpnp{abq}{1-\frac
    1{p^2}}\] for any $q \in \ZZp$. Since $\tau(\prodpd{b}{p})=2^{\omega(b)}$,
  Corollary~\ref{cor:function_small_has_average} gives another proof of
  \cite[Lemma~1]{arXiv:0710.1560}.
\end{example}

\section{Arithmetic functions in several variables}\label{sec:arithmetic_r}

Here, we are interested in the average size of certain arithmetic functions in
several variables when summing them over some or all of these variables. Our
goal is to characterize functions explicitly that typically appear in proofs
of Manin's conjecture, and to show that they lie in $\Theta_{2,r}(C)$ (see
Definition~\ref{def:arithmetic_function_r_summable}), so that we can apply
Proposition~\ref{prop:several_summations}.

\begin{definition}\label{def:arithmetic_function_r}
  Let $r \in \ZZnn$.  For any $\e_1, \dots, \e_r \in \ZZp$ and any prime
  $p$, we define
  \begin{equation*}
    \kk_p(\e_1, \dots, \e_r) = (k_1, \dots, k_r).
  \end{equation*}
  where $p^{k_i} \ddiv \e_i$ for $i=1, \dots, r$

  Let $\Theta_{3,0} = \RR$. For $r \in \ZZp$, let $\Theta_{3,r}$ be the set of
  all non-negative functions $\theta: \ZZp^r \to \RR$ for which there are
  non-negative functions $\theta_p : \ZZnn^r\to \RR$ for any prime $p$ such
  that \[\theta(\e_1, \dots, \e_r) = \prodp{\theta_p(\kk_p(\e_1, \dots,
    \e_r))}\] for all $\e_1, \dots, \e_r \in \ZZp$. We call the functions
  $\theta_p$ \emph{local factors} of $\theta$.

  For $\kk \in \ZZ^r$, we define
  \begin{equation*}
    \supp(\kk)=\{i \in \{1, \dots, r\}
    \where k_i \ne 0\}, \qquad \Sigma(\kk) = k_1+\dots+k_r.
  \end{equation*}
\end{definition}

\begin{definition}\label{def:arithmetic_function_r_small}
  Let $C \in \RR_{\ge 1}$. Let $\Theta_{4,0}(C)=\RR$. For any $r \in \ZZp$,
  let $\Theta_{4,r}(C)$ be the set of all functions $\theta \in \Theta_{3,r}$
  whose local factors $\theta_p$ fulfill the following conditions for any
  prime $p$.
  \begin{enumerate}[(1)]
  \item For any $\kk, \kk' \in \ZZnn^r$ with $\supp(\kk-\kk')=\{i\}$ and
    $\Sigma(\kk-\kk')=1$ (i.e., $\kk, \kk'$ differ by $1$ at the $i$-th
    coordinate $k_i, k'_i$ and coincide at all other coordinates),
    \begin{equation*}
      |\theta_p(\kk)-\theta_p(\kk')| \le
      \begin{cases}
        C, &k_i=1,\ \#\supp(\kk) \ge 2,\\
        Cp^{-k_i}, &\text{otherwise.}
      \end{cases}
    \end{equation*}
  \item For any $\kk \in \ZZp^r$, 
    \begin{equation*}
      \theta_p(\kk) \le
      \begin{cases}
        1+Cp^{-2}, &\kk = (0, \dots, 0),\\
        1+\#\supp(\kk)\cdot Cp^{-1}, &\text{otherwise}.
      \end{cases}
    \end{equation*}
  \end{enumerate}
\end{definition}

We recall Definition~\ref{def:function_small} of $\Theta_2$.

\begin{lemma}\label{lem:arithmetic_function_r_small_is_small}
  For $r \in \ZZp$, $C \in \RR_{\ge 1}$, let $\theta \in
  \Theta_{4,r}(C)$, with local factors $\theta_p$.
  As a function in $\e_r$,
  \begin{equation*}
    \theta \in \Theta_2\left(\prodpd{\e_1\cdots\e_{r-1}}{p}, C,
    C, (3rC)^{\omega(\e_1\cdots\e_{r-1})}
    \prodpp{1+\frac{C}{p^2}}\right).
  \end{equation*}
  The function $\theta':\ZZp^{r-1} \to \RR$ defined by
  \begin{equation*}
    \theta'(\e_1, \dots, \e_{r-1}) = \A(\theta(\e_1, \dots, \e_r),\e_r,1),
  \end{equation*}
  has local factors
  \begin{equation*}
    \theta'_p(\kk) = \left(1-\frac 1 p\right) \sum_{k_r=0}^\infty
    \frac{\theta_p(\kk,k_r)}{p^{k_r}}.
  \end{equation*}
\end{lemma}

\begin{proof}
  We have 
  \begin{equation*}
    \theta(\e_1, \dots, \e_r) = 
    \prodpdd{k_r}{\e_r}{\theta_p(\kk_p(\e_1,\dots,\e_{r-1}),k_r)}
    \prodpn{\e_r}{\theta_p(\kk_p(\e_1, \dots, \e_{r-1}),0)}.
  \end{equation*}
  Therefore, $\theta$ as a function in $\e_r$ lies in $\Theta_1$, with
  corresponding $c=1$ and $A_p(\nu) = \theta_p(\kk_p(\e_1, \dots,
    \e_{r-1}),\nu)$ for any $\nu \in \ZZnn$ and $p \in \P$.

  Now we check that $c,A_p$ fulfill the conditions of
  Definition~\ref{def:function_small}.
  For any $\kk \in \ZZnn^r$, $\theta_p(\kk)$ is at most
  \begin{equation*}
    \begin{split}
      &\theta_p((0, \dots, 0))\\ &+ \sum_{i=1}^r \sum_{n=1}^{k_i}
      |\theta_p(k_1, \dots, k_{i-1},n,0, \dots,
      0)-\theta_p(k_1, \dots, k_{i-1},n-1,0, \dots, 0)|\\
      &\le (1+Cp^{-2})+\sum_{i=1}^r\left(C+\sum_{n=2}^{k_i}Cp^{-n}\right)\\
      &\le 1+Cp^{-2}+r\left(C+\frac{C}{p^2(1-p^{-1})}\right)\\
      &\le 3rC.
    \end{split}
  \end{equation*}
  Therefore,
  \begin{equation*}
    |A_p(0)| \le
    \begin{cases}
      3rC, &p \div \e_1\cdots\e_{r-1},\\
      1+Cp^{-2}, &p \ndiv \e_1\cdots\e_{r-1},
    \end{cases}
  \end{equation*}
  so that, for any $k \in \ZZp$,
  \begin{equation*}
    \Big|c \prodpn{k}{A_p(0)} \Big| \le (3rC)^{\omega(\e_1\cdots\e_{r-1})}
    \prodpp{1+\frac{C}{p^2}}.
  \end{equation*}
  Furthermore, for any prime $p$ and $\nu \ge \ZZp$,
  \begin{multline*}
    |A_p(\nu)-A_p(\nu-1)| = |\theta_p(\kk_p(\e_1, \dots,
    \e_{r-1}),\nu)-\theta_p(\kk_p(\e_1, \dots, \e_{r-1}),\nu-1)|\\ \le
    \begin{cases}
      C, &\nu=1,\ \#\supp(\kk_p(\e_1, \dots, \e_{r-1})) > 0,\\
      Cp^{-\nu}, &\text{otherwise,}
    \end{cases}
  \end{multline*}
  where the first case applies if and only if $p^\nu \div
  \prodpd{\e_1\cdots\e_{r-1}}{p}$.

  Therefore, we may define $\theta'$ as in the statement of the
  lemma. By definition,
  \begin{equation*}
    \theta'(\e_1, \dots, \e_{r-1})
    =\prodpp{\left(1-\frac 1 p\right) \sum_{k_r =0}^\infty \frac{
        \theta_p(\kk_p(\e_1,\dots,\e_{r-1}),k_r)}{p^{k_r}}}
  \end{equation*}
  for any $\e_1,\dots, \e_{r-1}$. Here, we can read off local factors for
  $\theta'$ as claimed.
\end{proof}

\begin{lemma}\label{lem:arithmetic_function_r_small_induction}
  Let $r,C,\theta,\theta'$ be as in
  Lemma~\ref{lem:arithmetic_function_r_small_is_small}.
  Then $\theta' \in \Theta_{4,r-1}(3C)$.
\end{lemma}

\begin{proof}
  By Lemma~\ref{lem:arithmetic_function_r_small_is_small}, local factors
  of $\theta'$ are
  \begin{equation*}
    \theta'_p(\kk) = \left(1-\frac 1 p\right) \sum_{k_r=0}^\infty
    \frac{\theta_p(\kk,k_r)}{p^{k_r}}.
  \end{equation*}
  For $k_r \in \ZZp$, we have 
  \begin{equation*}
    |\theta_p(0,\dots,0,k_r)-\theta_p(0, \dots, 0,0)| \le
      \sum_{n=1}^{k_r} \frac{C}{p^n} \le \frac{2C}{p}.
  \end{equation*}
  Therefore,
  \begin{multline*}
    |\theta'_p(0, \dots, 0)-\theta_p(0, \dots, 0,0)|\\ \le
    \left(1-\frac 1 p\right) \sum_{k_r=1}^\infty \frac{|\theta_p(0, \dots,
        0,k_r) - \theta_p(0, \dots, 0,0)|}{p^{k_r}} \le
    \frac{2C}{p^2}.
  \end{multline*}
  By the assumption on $\theta_p(0, \dots, 0)$, this implies $\theta'_p(0,
  \dots, 0) \le 1+3Cp^{-2}$.
  
  For $\kk \in \ZZnn^{r-1} \setminus \{(0, \dots, 0)\}$, so that
  $\#\supp(\kk)+1 \le 2\#\supp(\kk)$, we have
  \begin{equation*}
    \theta'_p(\kk) \le \left(1-\frac 1 p\right)
    \sum_{k_r=0}^\infty \frac{1+(1+\#\supp(\kk))Cp^{-1}}{p^{k_r}} \le
    1+\frac{\#\supp(\kk)\cdot 2C}{p}.
  \end{equation*}

  Now we consider $\kk, \kk'\in \ZZnn^{r-1}$ with $\supp(\kk-\kk')=\{i\}$ and
  $\Sigma(\kk-\kk')=1$, so that we have $k_i=k'_i+1$ for the $i$-th
  coordinates $k_i,k'_i$ of $\kk,\kk'$. We have
  \begin{equation*}
    |\theta_p'(\kk)-\theta_p'(\kk')| \le \left(1-\frac 1
      p\right) \sum_{k_r=0}^\infty \frac{|\theta_p(\kk,k_r) -
      \theta_p(\kk',k_r)|}{p^{k_r}}.
  \end{equation*}
  If $k_i \ge 2$, then
  \begin{equation*}
    |\theta_p'(\kk)-\theta_p'(\kk')| \le \frac{C}{p^{k_i}}.
  \end{equation*}
  If $k_i = 1$ and $\#\supp(\kk)=1$, then
  \begin{equation*}
    |\theta_p'(\kk)-\theta_p'(\kk')| \le \left(1-\frac 1
      p\right) \left(\frac{C}{p}+\sum_{k_r=1}^\infty
      \frac{C}{p^{k_r}}\right) \le \frac{2C}{p}.
  \end{equation*}
  If $k_i=1$ and $\#\supp(\kk)\ge 2$, then
  \begin{equation*}
    |\theta_p'(\kk)-\theta_p'(\kk')| \le C.
  \end{equation*}
  This completes the proof.
\end{proof}

Recall Definition~\ref{def:function_bound} of $\Theta_{0,r}(C)$,
Definition~\ref{def:arithmetic_function_r_with_average} of
$\Theta_{1,r}(C,\e_r)$ and Definition~\ref{def:arithmetic_function_r_summable}
of $\Theta_{2,r}(C)$.

\begin{cor}\label{cor:arithmetic_function_r_small_has_average}
  For any $r \in \ZZnn$, $C \in \ZZnn$, we have
  \begin{equation*}
    \Theta_{4,r}(C) \subset
    \Theta_{0,r}(0) \cap \Theta_{1,r}(12rC^2,\e_r) \cap \Theta_{2,r}(12r(3^rC)^2).
\end{equation*}
\end{cor}

\begin{proof}
  We prove the results by induction on $r$. The case $r=0$ is trivial. Let $r
  \in \ZZp$ and $\theta \in \Theta_{4,r}(C)$. 

  Since
  \begin{equation*}
    \theta(\e_1, \dots, \e_r) \le \prod_{i=1}^r (\phid(\e_i))^C
    \prodpp{1+\frac{C}{p^{2}}},
  \end{equation*}
  for any $\e_1, \dots, \e_r \in \ZZp$, we have $\theta \in \Theta_{0,r}(0)$
  (cf. Example~\ref{exa:phis_phid_omega}).

  By Lemma~\ref{lem:arithmetic_function_r_small_is_small} and
  Corollary~\ref{cor:function_small_has_average}, $\theta \in \Theta_0(C)$ as
  a function in $\e_r$. We define
  \begin{equation*}
    \begin{split}
      \theta'(\e_1, \dots, \e_{r-1}) &= \A(\theta(\e_1, \dots, \e_r),\e_r),\\
      \theta''(\e_1, \dots, \e_{r-1}) &= \E(\theta(\e_1, \dots, \e_r),\e_r).
    \end{split}
  \end{equation*}
  By Lemma~\ref{lem:arithmetic_function_r_small_induction}, we have $\theta'
  \in \Theta_{4,r-1}(3C)$. By induction, $\theta' \in \Theta_{0,r-1}(0)$. By
  Corollary~\ref{cor:function_small_has_average}, \[\theta''(\e_1, \dots,
  \e_{r-1}) = O_C((12rC^2)^{\omega(\e_1\cdots\e_{r-1})})\] since
  $\tau(\prodpd{n}{p}) = 2^{\omega(n)}$ for any $n \in \ZZp$. By
  Example~\ref{exa:phis_phid_omega}, $\theta'' \in \Theta_{0,r-1}(12rC^2)$.
  Therefore, $\theta \in \Theta_{1,r}(12rC^2, \e_r)$.

  Since $\theta' \in \Theta_{2,r-1}(12(r-1)(3^{r-1}(3C))^2)$ by induction, this
  implies $\theta \in \Theta_{2,r}(12r(3^rC)^2)$.
\end{proof}

\begin{lemma}\label{lem:arithmetic_function_r_average}
  Let $r \in \ZZp$ and $\theta_r \in \Theta_{4,r}(C)$, with local factors
  $\theta_{r,p}$. Let $\ell \in \{0, \dots, r-1\}$.  Local factors of
  $\theta_\ell = \A(\theta_r(\e_1, \dots, \e_r), \e_r, \dots, \e_{\ell+1})$
  are given by
  \begin{equation*}
    \theta_{\ell,p}(\kk) = \left(1-\frac 1 p\right)^{r-\ell} \sum_{\kk' \in
      \ZZnn^{r-\ell}} \frac{\theta_{r,p}(\kk,\kk')}{p^{\Sigma(\kk')}}.
  \end{equation*}
  In particular, for $\theta_0 = \A(\theta_r(\e_1, \dots, \e_r),\e_r, \dots,
  \e_1) \in \RR$, we have
  \begin{equation*}
    \theta_{0} = \prodpp{\left(1-\frac 1 p\right)^r \sum_{\kk \in \ZZnn^r}
      \frac{\theta_{r,p}(\kk)}{p^{\Sigma(\kk)}}}.
  \end{equation*}
\end{lemma}

\begin{proof}
  We prove the claim by induction on $\ell$. Local factors of $\theta_{r-1}$
  are given by Lemma~\ref{lem:arithmetic_function_r_small_is_small}. By an
  application of Lemma~\ref{lem:arithmetic_function_r_small_is_small} to
  $\theta_\ell \in \Theta_{4,\ell}(3^{r-\ell}C)$
  (Lemma~\ref{lem:arithmetic_function_r_small_induction}) and the induction
  hypothesis, local factors of $\theta_{\ell-1}$ are
  \begin{equation*}
    \begin{split}
      \theta_{\ell-1,p}(\kk) &= \left(1-\frac 1 p\right)
      \sum_{k_\ell=0}^\infty
      \frac{\theta_{\ell,p}(\kk,k_\ell)}{p^{k_\ell}}\\
      &=\left(1-\frac 1 p\right)^{r-(\ell-1)} \sum_{k_\ell =0}^\infty \frac
      1{p^{k_\ell}} \sum_{\kk' \in \ZZnn^{r-\ell}}
      \frac{\theta_{r,p}(\kk,k_\ell,\kk')}{p^{\Sigma(\kk')}}\\
      &=\left(1-\frac 1 p\right)^{r-(\ell-1)} \sum_{\kk'' \in
          \ZZnn^{r-(\ell-1)}}
        \frac{\theta_{r,p}(\kk,\kk'')}{p^{\Sigma(\kk'')}}
    \end{split}
  \end{equation*}
  This completes the induction step.
\end{proof}

In many applications, we are concerned with a function $\theta \in
\Theta_{3,r}$ whose local factors $\theta_p(\kk)$ only depend on
$\supp(\kk)$. In this case, the notation and results can be simplified as
follows.

\begin{definition}\label{def:arithmetic_function_r_simple}
  Let $\Theta_{3,0}' = \RR$. For $r \in \ZZp$, let $\Theta_{3,r}'$ be the set
  of all $\theta \in \Theta_{3,r}$, with local factors $\theta_p$, such that,
  for any $\kk, \kk' \in \ZZnn^r$ with $\supp(\kk) = \supp(\kk')$, we have
  $\theta_p(\kk) = \theta_p(\kk)$.

  Let $\theta \in \Theta_{3,r}'$ with local factors $\theta_p$. For any $I
  \subset \{1, \dots, r\}$, we define $\theta_p(I)$ as $\theta_p(\kk_I)$ for
  any $\kk_I \in \ZZp^r$ with $\supp(\kk_I) = I$.

  For any $\e_1, \dots, \e_\ell \in \ZZ$, let
  \begin{equation*}
    I_p(\e_1, \dots, \e_r) = \supp(\kk_p(\e_1, \dots, \e_r)) = \{i \in
    \{1, \dots, r\} \where p \div \e_i\},
  \end{equation*}
  so that
  \begin{equation*}
    \theta(\e_1, \dots, \e_r) = \prod_p \theta_p(I_p(\e_1, \dots, \e_r)).
  \end{equation*}
\end{definition}

\begin{definition}\label{def:arithmetic_function_r_simple_small}
  Let $r \in \ZZp$ and $C \in \RR_{\ge 1}$. Let $\Theta_{4,r}'(C)$ be the set
  of all $\theta \in \Theta_{2,r}'$ such that, for any $I \subset \{1, \dots,
  r\}$ and $p \in \P$,
  \begin{equation*}
    |\theta_p(I) - 1| \le
    \begin{cases}
      Cp^{-2}, &\#I=0,\\
      Cp^{-1}, &\#I=1,\\
      C, &\#I\ge 2
    \end{cases}
  \end{equation*}
  and $\theta_p(I) \le 1+\#I\cdot Cp^{-1}$ if $\#I > 0$.
\end{definition}

\begin{cor}\label{cor:arithmetic_function_r_simple_small_is_small}
  For any $r \in \ZZp$ and $C \in \RR_{\ge 1}$, we have
  \begin{equation*}
    \Theta_{4,r}'(C) \subset \Theta_{4,r}(2C) \subset \Theta_{0,r}(0) \cap
    \Theta_{1,r}(48rC^2,\e_r) \cap \Theta_{2,r}(48r(3^rC)^2).
\end{equation*}
\end{cor}

\begin{proof}
  Let $\theta \in \Theta_{4,r}'(C)$. Let $\kk, \kk' \in \ZZnn^r$ with
  $\supp(\kk-\kk') = \{i\}$ and $\Sigma(\kk-\kk') = 1$. If $k_i \ge 2$, then
  $\supp(\kk) = \supp(\kk')$, so that
  $\theta_p(\kk)=\theta_p(\kk')$. If $k_i=1$, then $\#\supp(\kk) =
  \#\supp(\kk')+1$, so that
  \begin{multline*}
    |\theta_p(\kk)-\theta_p(\kk')| =
    |\theta_p(\supp(\kk))-\theta_p(\supp(\kk'))| \\
    \le
    \begin{cases}
      2C, &\#\supp(\kk) \ge 2,\\
      2Cp^{-1}, &\#\supp(\kk) =1.
    \end{cases}
  \end{multline*}
  Furthermore, for any $\kk \in \ZZnn^r$,
  \begin{equation*}
    \theta_p(\kk) = \theta_p(\supp(\kk)) \le
    \begin{cases}
      1+Cp^{-2}, &\kk = (0, \dots, 0),\\
       1+\#\supp(\kk) \cdot C
      p^{-1}, &\text{otherwise.}
    \end{cases}
  \end{equation*}
  This shows that $\theta \in \Theta_{4,r}(2C)$, and the result follows from
  Corollary~\ref{cor:arithmetic_function_r_small_has_average}.
\end{proof}

\begin{cor}\label{cor:arithmetic_function_r_simple_average}
  Let $r \in \ZZp$ and $\theta_r \in \Theta_{4,r}'$. Let $\ell \in \{0, \dots,
  r-1\}$. The function $\theta_\ell$ defined by $\theta_\ell(\e_1, \dots,
  \e_\ell) = \A(\theta_r(\e_1, \dots, \e_r),\e_r, \dots, \e_{\ell+1})$ has
  local factors $\theta_{\ell,p}$ given by
  \begin{equation*}
    \theta_{\ell,p}(I) = \sum_{J \subset \{\ell+1, \dots, r\}}
    \left(1-\frac 1 p\right)^{r-\ell-\#J} \left(\frac 1
      p\right)^{\#J}\theta_{r,p}(I\cup J),
  \end{equation*}
  for any $I \subset \{1, \dots, \ell\}$.  In particular,   
  \begin{equation*}
    \theta_0 = \prod_p \sum_{J \subset \{1, \dots, r\}}
    \left(1-\frac 1 p\right)^{r-\#J} \left(\frac 1
      p\right)^{\#J}\theta_{r,p}(J),
  \end{equation*}
  while $\A(\theta_r(\e_1, \dots, \e_r),\e_r)$ has local factors
  \begin{equation*}
    \theta_{r-1,p}(I) = \left(1-\frac 1 p\right)\theta_{r,p}(I) + \frac 1 p
    \theta_{r,p}(I\cup\{r\}).
  \end{equation*}
\end{cor}

\begin{proof}
  This is a special case of Lemma~\ref{lem:arithmetic_function_r_average},
  which we may apply because of
  Corollary~\ref{cor:arithmetic_function_r_simple_small_is_small}.
\end{proof}

\section{Application to a quartic del Pezzo surface}\label{sec:qa1a3}

Let $S \subset \Pfour$ be the quartic del Pezzo surface defined by
\[x_0^2+x_0x_3+x_2x_4 = x_1x_3-x_2^2 = 0.\] It contains exactly two
singularities, namely $(0:0:0:0:1)$ of type $\Athree$ and $(0:1:0:0:0)$ of
type $\Aone$, and three lines, \[\{x_0=x_1=x_2=0\},\quad
\{x_0+x_3=x_1=x_2=0\},\quad \{x_0=x_2=x_3=0\}.\]

\begin{theorem}\label{thm:qa1a3_main}
  We have
  \[N_{U,H}(B) = \alpha(\tS)\left(\prod_p \omega_p\right) \omega_\infty B(\log
  B)^5 + O(B(\log B)^4(\log \log B)^2)\] for $B \ge 3$, where
  \begin{equation*}
    \begin{split}
      \alpha(\tS) &= \frac{1}{8640},\\
      \omega_p & = \left(1-\frac 1 p\right)^6 \left(1+\frac 6 p +\frac 1
        {p^2}\right) ,\\
      \omega_\infty & =
      \int_{|x_0|,|x_2|,|x_2^2/x_1|,|(x_0^2x_1+x_0x_2^2)/(x_1x_2)|\le 1,\ 0\le
        x_1 \le 1} \frac{1}{x_1x_2} \dd x_0 \dd x_1 \dd x_2.
    \end{split}
  \end{equation*}
\end{theorem}

\begin{remark}
  We note that $S$ is not an equivariant compactification of the additive
  group $\Ga^2$, so that Theorem~\ref{thm:qa1a3_main} does not follow from the
  general results of \cite{MR1906155}.

  Indeed, the projection $S \rto \Ptwo$ from the line $\{x_0=x_1=x_2=0\}$ is
  an isomorphism between the complement $U$ of the three lines in $S$ and the
  complement of two lines in $\Ptwo$. If $S$ were an equivariant
  compactification of $\Ga^2$, then there would be a $\Ga^2$-structure on
  $\Ptwo$ fixing two lines, contradicting \cite[Proposition~3.2]{MR1731473}.
\end{remark}

Since all lines on $S$ are defined over $\QQ$, the minimal desingularization
$\tS$ of $S$ is the blow-up of $\Ptwo$ in five rational points, so that
$\Pic(\tS) \cong \ZZ^6$.  The effective cone in $\Pic(\tS)_\RR = \Pic(\tS)
\otimes_\ZZ \RR \cong \RR^6$ of $\tS$ has seven generators. The investigation
of the geometry of $\tS$ in \cite[Section~7]{math.AG/0604194} shows the
intersection of its dual (with respect to the intersection form $(\cdot,
\cdot)$ on $\Pic(\tS)_\RR$) with the hyperplane $\{\tt \in \Pic(\tS)_\RR
\where (\tt,-K_\tS)=1\}$ is the polytope
\begin{equation}\label{eq:qa1a3_P}
  \begin{split}
    P&=\left\{(t_1, \dots, t_6) \in \RRnn^6 \Where
      \begin{aligned}
        &t_1+t_2+t_3-2t_5-t_6\ge 0,\\
        &2t_1+2t_2+3t_3+2t_4+t_6=1
      \end{aligned}
    \right\}\\
    \cong P' &=\left\{(t_1, \dots, t_5) \in \RRnn^5 \Where
      \begin{aligned}
        &2t_1+2t_2+3t_3+2t_4 \le 1, \\
        &3t_1+3t_2+4t_3+2t_4-2t_5 \ge 1
      \end{aligned}
    \right\}
  \end{split}  
\end{equation}

We check that Theorem~\ref{thm:qa1a3_main} agrees with the conjectures of
Yu.~I.~Manin \cite{MR89m:11060} and E.~Peyre \cite{MR1340296} that predict an
asymptotic formula with main term $cB(\log B)^k$, where $k = \rk \Pic(\tS)-1$
and $c$ is the the product of local densities and $\vol(P)$. Indeed, $\rk
\Pic(\tS) = 6$ since $S$ is split. By a computation as in
\cite[Lemma~1]{MR2320172}, $\omega_p$ resp.\ $\omega_\infty$ as in the
statement of Theorem~\ref{thm:qa1a3_main} agree with the density of $p$-adic
resp.\ real points on $S$. Finally, \[\vol(P) = \vol(P') = \alpha(\tS) =
\frac{1/180}{\#W(\Aone)\cdot \#W(\Athree)}= \frac{1}{8640}\] by
\cite[Theorem~4]{MR2318651} and \cite[Theorem~1.3]{MR2377367}, where
$W(\bA_i)$ is the Weyl group of the root system $\bA_i$.

\subsection{Passage to a universal torsor}

We carry out step \eqref{it:strategy_passage} of the strategy described in
Section~\ref{sec:counting_introduction}. Let
\begin{equation*}
  \ee = (\e_1, \dots, \e_7), \quad \ee' = (\e_1, \dots, \e_8), \quad \ee'' = (\e_1, \dots, \e_9), \quad
  \ee^\kk = \e_1^{k_1}\cdots\e_7^{k_7},
\end{equation*}
for any $\kk = (k_1, \dots, k_7) \in \RR^7$.
For $i = 1, \dots, 9$, let 
\begin{equation}\label{eq:qa1a3_ranges}
  (\ZZ_i, J_i, J_i')=
  \begin{cases}
    (\ZZp, \RR_{\ge 1}, \RR_{\ge 1}), & i \in \{1, \dots, 5\},\\
    (\ZZp, \RR_{\ge 1}, \RRnn), & i =6,\\
    (\ZZnz, \RR_{\le -1} \cup \RR_{\ge 1}, \RR) , & i=7,\\
    (\ZZ, \RR,\RR), & i \in \{8, 9\}.
  \end{cases}
\end{equation}

\begin{figure}[ht]
  \centering
  \[\xymatrix{E_9 \ar@{-}[dr] \ar@{-}[dd] \ar@{-}[rrrr]& & & & E_1
    \ar@{-}[dr]\\
    & E_7 \ar@{-}[r] & E_5 \ar@{-}[r] & E_6 \ar@{-}[r] & E_4 \ar@{-}[r] & E_3\\
    E_8 \ar@{-}[ur] \ar@{-}[rrrr] & & & & E_2 \ar@{-}[ur]}\]
  \caption{Configuration of curves on $\tS$.}
  \label{fig:qa1a3_dynkin}
\end{figure}

The following result is based on our investigation
\cite[Section~7]{math.AG/0604194} of \[\Cox(\tS)=\QQ[\e_1, \dots,
\e_9]/(\e_1\e_9+\e_2\e_8+\e_4\e_5^3\e_6^2\e_7),\] where $\TtS$ an open subset
of $\Spec(\Cox(\tS))$. It is derived using the method developed in
\cite[Section~4]{MR2290499}. Figure~\ref{fig:qa1a3_dynkin} shows the
configuration of curves $E_1, \dots, E_9$ on $\tS$ that correspond to the
generators $\e_1, \dots, \e_9$ of $\Cox(\tS)$, with edges between pairs of
intersecting curves. Here, $E_1, E_2, E_5$ are strict transforms of the three
lines $\{x_0+x_3=x_1=x_2=0\}$, $\{x_0=x_1=x_2=0\}$, $\{x_0=x_2=x_3=0\}$, while
$E_3, E_4, E_6$ and $E_7$ are the exceptional divisors obtained by blowing up
the $\Athree$ and $\Aone$ singularities.

\begin{lemma}\label{lem:qa1a3_bijection}
  The map $\psi: \TtS \to S$ defined by
  \begin{equation*}
    \ee'' \mapsto(\base 0 1 1 1 1 1 1\e_8, \base 2 2 3
    2 0 1 0, \base 1 1 2 2 2 2 1, \base 0 0 1 2 4 3 2, \e_7\e_8\e_9)
  \end{equation*}
  induces a bijection $\Psi$ between \[T_0(B)=\{\ee'' \in \ZZ_1 \times \dots
  \times \ZZ_9 \where \text{(\ref{eq:qa1a3_torsor}), (\ref{eq:height}),
    (\ref{eq:cpe}) hold}\}\] and $\{\xx \in U(\QQ) \mid H(\xx) \le B\}$, where
  \begin{gather}
    \label{eq:qa1a3_torsor}
    \e_1\e_9+\e_2\e_8+\e_4\e_5^3\e_6^2\e_7 = 0,\\
    \label{eq:height}
    \max_{i \in \{0, \dots, 4\}}|\Psi(\ee'')_i| \le B,\\
    \label{eq:cpe} \text{$\e_1, \dots, \e_9$ fulfill coprimality conditions
      as in Figure~\ref{fig:qa1a3_dynkin}.}
  \end{gather}
\end{lemma}

Using (\ref{eq:qa1a3_torsor}) to eliminate $\e_9$, the height condition
(\ref{eq:height}) is equivalent to $h(\ee';B) \le 1$, where
\begin{equation*}
  h(\ee';B) = B^{-1} \max\left\{
    \begin{aligned}
      &|\base 0 1 1 1 1 1 1\e_8|, |\base 2 2 3 2 0 1 0|, |\base 1 1 2 2 2 2
      1|,\\ &|\base 0 0 1 2 4 3 2|,
      |\e_1^{-1}(\e_2\e_7\e_8^2+\e_4\e_5^3\e_6^2\e_7^2\e_8)|
    \end{aligned}
  \right\}.
\end{equation*}

\subsection{Counting points}

We come to step \eqref{it:strategy_counting} of our strategy. We recall the
definition (\ref{eq:qa1a3_ranges}) of $J_1, \dots, J_8$ and define \[\R(B)=
\{\ee' \in J_1 \times \dots \times J_8 \mid h(\ee';B) \le 1\}.\] Using the
results of Sections~\ref{sec:first_summation}, \ref{sec:generic_completion}
and \ref{sec:arithmetic_r}, we show (Lemma~\ref{lem:qa1a3_summation_rest})
that the number of integral points in the region $\R(B)$ on $\TtS$ that
satisfy the coprimality conditions~(\ref{eq:cpe}) can be approximated by the
product of the volume of $\R(B)$ and $p$-adic densities coming from the
coprimality conditions.

\begin{lemma}\label{lem:sum_a1a3_first}
  We have \[N_{U,H}(B) = \sum_{\ee \in \ZZ_1 \times \dots \times \ZZ_7}
  \theta_1(\ee)V_1(\ee;B) + O(B(\log B)^2),\] where
  \[V_1(\ee;B)=\int_{\ee' \in \R(B)} \e_1^{-1} \dd \e_8\]
  and, in the notation of Definition~\ref{def:arithmetic_function_r_simple},
  \begin{equation*}
    \theta_1(\ee) = \prodp{\theta_{1,p}(I_p(\ee))}\] with $I_p(\ee) = \{i \in \{1, \dots, 7\} \where p
    \div \e_i\}$ and  \[\theta_p(I)=
  \begin{cases}
    1, & I = \emptyset, \{1\}, \{2\}, \{7\},\\
    1-\frac 1 p, & I = \{4\}, \{5\}, \{6\}, \{1,3\}, \{2,3\}, \{3,4\},
    \{4,6\}, \{5,6\}, \{5,7\},\\
    1-\frac 2 p, & I = \{3\},\\
    0, & \text{all other $I \subset \{1, \dots, 7\}$.}
  \end{cases}
\end{equation*}
\end{lemma}

\begin{proof}
  By Lemma~\ref{lem:qa1a3_bijection}, our counting problem has the special
  form of Section~\ref{sec:first_summation}. Table~\ref{tab:dict} provides a
  dictionary between the notation of Section~\ref{sec:first_summation} and the
  present situation.

  \begin{table}[ht]
    \begin{center}
      \begin{equation*}
        \begin{array}{|r|l||r|l|}
          \hline
          (r,s,t) & (3,1,1)&  \delta & \eta_3\\
          \hline 
          (\al_0;\al_1,\ldots,\al_r) & (\e_7;\e_4,\e_6,\e_5) &
          (a_0;a_1,\ldots,a_r) & (1;1,2,3) \\
          \hline
          (\beta_0;\beta_1,\ldots,\beta_s) & (\e_8;\e_2) &
          (b_0;b_1,\ldots,b_s) & (1;1) \\
          \hline
          (\gamma_0;\gamma_1,\ldots,\gamma_t) & (\e_9;\e_1) &
          (c_1,\ldots,c_t) & (1,1) \\
          \hline
          \Pi(\aa) & \e_4\e_5^3\e_6^2 & \Pi'(\delta,\aa)) & \e_3\e_4\e_6\\
          \hline
          \Pi(\bb) & \e_2 & \Pi'(\delta,\bb)) & \e_3\\
          \hline
          \Pi(\cc) & \e_1 & \Pi'(\delta,\cc)) & \e_3\\
          \hline
        \end{array}
      \end{equation*}
    \end{center}
    \caption{Application of Proposition~\ref{prop:generic_first_step}.}
    \label{tab:dict}
  \end{table}

  By Proposition~\ref{prop:generic_first_step},
  \[N_{U,H}(B) = \sum_{\ee \in \ZZ_1 \times \dots \times \ZZ_7}
  (\theta_1(\ee)V_1(\ee;B)+R_1(\ee;B)),\] where local factors of $\theta_1$ as
  in the statement of Proposition~\ref{prop:generic_first_step} are easily
  computed to be the ones in the statement of this lemma, and \[R_1(\ee;B) \ll
  2^{\omega(\e_3)+\omega(\e_3\e_4\e_5\e_6)}.\] Both $N_1$ and $V_1$ and
  therefore also $R_1$ vanish unless $|\base 1 1 2 2 2 2 1| \le B$, so
  \begin{equation*}
  \begin{split}
    \sum_{\ee} R_1(\ee;B) \ll{}&\sum_{\ee}
    2^{\omega(\e_3)+\omega(\e_3\e_4\e_5\e_6)}\\
    \ll{}& \sum_{\e_1, \dots, \e_6}
    \frac{2^{\omega(\e_3)+\omega(\e_3\e_4\e_5\e_6))}B}{\base 1 1 2
      2 2 2 0}\\
    \ll{}& B(\log B)^2.
  \end{split}
\end{equation*}
This completes the proof.
\end{proof}

\begin{lemma}\label{lem:qa1a3_summation_rest}
  We have
  \[N_{U,H}(B) = \left(\prodp{\omega_p}\right) V_0(B)+O(B(\log B)^4(\log \log
  B)^2),\]
  where
  \begin{equation*}
    V_0(B)=\int_{\ee} V_1(\ee;B) \dd \ee = \int_{\ee' \in \R(B)} \e_1^{-1} \dd
      \ee'.
  \end{equation*}
\end{lemma}

\begin{proof}
  Clearly $\theta_1 \in \Theta_{4,7}'(2)$, so  $\theta_1 \in \Theta_{2,7}(C)$
  for some $C \in \ZZp$ by
  Corollary~\ref{cor:arithmetic_function_r_simple_small_is_small}.
  By Lemma~\ref{lem:generic_bounds}\eqref{it:generic_bounds_b},
  \begin{multline*}
    V_1(\ee;B) \ll \frac{B^{1/2}}{\e_1^{1/2}\e_2^{1/2}|\e_7|^{1/2}} \\
    = \frac{B}{|\base 1 1 1 1 1 1 1|} \cdot \left(\frac{B}{|\base 2 2 3 2 0
        1 0|}\right)^{-1/4}\left(\frac{B}{|\base 0 0 1
        2 4 3 2|}\right)^{-1/4}.
  \end{multline*}
  As $V_1(\ee;B) = 0$ unless $1 \le \e_1, \dots, \e_7 \le B$ and $|\base
  2232010| \le B$ and $|\base 0 0 1 2 4 3 2| \le B$, we can apply
  Proposition~\ref{prop:several_summations} with $(r,s)=(5,2)$, $a_1=a_2=1/4$,
  \begin{equation*}
    (k_{i,j})_{\substack{1 \le i \le 7\\ 1 \le j \le 2}} =
    \begin{pmatrix}
      2 & 2 & 3 & 2 & 0 & 1 & 0\\
      0 & 0 & 1 & 2 & 4 & 3 & 2
    \end{pmatrix}.
  \end{equation*}

  We compute
  \begin{equation*}
    \A(\theta_1(\ee), \e_7, \dots, \e_1)= \prodp{\left(1-\frac 1
        p\right)^6\left(1+\frac 6 p +\frac 1 {p^2}\right)} = \prodp{\omega_p}
  \end{equation*}
  using Corollary~\ref{cor:arithmetic_function_r_simple_average}.
\end{proof}

\subsection{The expected leading constant}

We carry out step~\eqref{it:strategy_manipulation} of our strategy. This step
is necessary as Lemma~\ref{lem:qa1a3_omega_infty} shows that the main term in
Theorem~\ref{thm:qa1a3_main} is obtained by replacing the integral over
$\R(B)$ by an integral over a region $\R'(B)$ that is closely related to the
shape of the polytope $P'$ (\ref{eq:qa1a3_P}).  Recalling
(\ref{eq:qa1a3_ranges}), we define
\begin{equation*}
  \begin{split}
    &\R_1'(B)=\{(\e_1, \dots, \e_5) \in J'_1\times \dots \times J'_5 \mid
    \e_1^2\e_2^2\e_3^3\e_4^2 \le B,\ \e_1^3\e_2^3\e_3^4\e_4^2\e_5^{-2}
    \ge B\},\\
    &\R_2'(\e_1, \dots, \e_5;B) = \{(\e_6,\e_7,\e_8) \in J'_6\times J'_7
    \times J'_8 \mid h(\e_1, \dots,\e_8;B) \le B\},\\
    &\R'(B)= \{(\e_1, \dots,\e_8) \in \RR^8 \mid (\e_1, \dots, \e_5) \in
    \R_1'(B), (\e_6,\e_7,\e_8) \in \R_2'(\ee;B)\}
  \end{split}
\end{equation*}
and
\begin{equation*}
  V_0'(B) = \int_{\ee' \in \R'(B)} \e_1^{-1} \dd \ee'.
\end{equation*}

\begin{lemma}\label{lem:qa1a3_omega_infty}
  We have
  \begin{equation*}
     V_0'(B) = \alpha(\tS)\omega_\infty
    B(\log B)^5.
  \end{equation*}
\end{lemma}

\begin{proof}
  By substituting
  \begin{equation*}
    x_1=B^{-1}\base 2 2 3 2 0 1 0, \ x_2=B^{-1}\base 1 1 2 2 2 2 1, \
    x_0=B^{-1}\base 0 1 1 1 1 1 1\e_8
  \end{equation*}
  into the expression for $\omega_\infty$ given in the statement of
  Theorem~\ref{thm:qa1a3_main}, we prove
  \begin{equation*}
    \frac{B \omega_\infty}{\e_1\cdots\e_5} = \int_{(\e_6,\e_7,\e_8) \in
      \R_2'(\e_1, \dots, \e_5;B)} \e_1^{-1} \dd \e_6 \dd \e_7 \dd \e_8.
  \end{equation*}
  
  Substituting $t_i=\frac{\log \e_i}{\log B}$ into $\alpha(\tS) = \vol(P') =
  \int_{\tt \in P'} \dd \tt$ shows
  \begin{equation*}
    \alpha(\tS)(\log B)^5 = \int_{\R_1'(B)} \frac{1}{\e_1\cdots\e_5} \dd
    \e_1\cdots \dd \e_5.
  \end{equation*}
  This completes the proof.
\end{proof}

\begin{lemma}\label{lem:qa1a3_manipulation}
  We have \[V_0(B) = V_0'(B) + O(B(\log B)^4).\]
\end{lemma}

\begin{proof}
  We define 
  \begin{equation*}
    V^{(i)}(B) = \int_{h(\ee'; B) \le 1,\ (\ee',\e_8) \in \R_i(B)} \e_1^{-1} \dd \ee',
  \end{equation*}
  where
  \begin{equation*}
    \begin{split}
      \R_0(B)&= \{\ee' \in J_1' \times \dots \times J_8' \where \e_6,
      |\e_7| \ge 1\},\\
      \R_1(B)&= \{\ee' \in J_1' \times \dots \times J_8' \where \e_6,
      |\e_7| \ge 1,\ \base 2232000 \le B\},\\
      \R_2(B)&= \left\{\ee' \in J_1' \times \dots \times J_8' \Where
        \begin{aligned}
          &\e_6, |\e_7| \ge 1,\\ &\base 2232000 \le B,\ \base 3342{-2}00 \ge B
        \end{aligned}
      \right\},\\
      \R_3(B)&= \{\ee' \in J_1' \times \dots \times J_8' \where \e_6
      \ge 1,\ \base 2232000 \le B,\ \base 3342{-2}00 \ge B\},\\
      \R_4(B)&= \{\ee' \in J_1' \times \dots \times J_8' \where \base
      2232000 \le B,\ \base 3342{-2}00 \ge B\}.
    \end{split}
  \end{equation*}
  For $i \in \{0, \dots, 3\}$, we will show that \[|V^{(i)}(B) - V^{(i+1)}(B)|
  \le \int_{\ee' \in (\R_i(B) \cup \R_{i+1}(B)) \setminus (\R_i(B) \cap
    \R_{i+1}(B)),\ h(\ee';B) \le 1} \e_1^{-1} \dd \ee'\] is $O(B(\log B)^4)$.
  Since $V_0(B) = V^{(0)}(B)$ and $V_0'(B) = V^{(4)}(B)$, this proves the
  result.

  For $i=0$, we note that $h(\ee',\e_8;B)\le 1$ and $\e_6 \ge 1$ imply $\base
  2232000 \le B$. Therefore, $V^{(0)}(B) = V^{(1)}(B)$.

  For $i=1$, we note that $\ee' \in \R_1(B) \setminus \R_2(B)$ implies $\e_5^2 >
  \base 3 3 4 2 0 0 0/B$ and $1 \le \e_1,\e_2,\e_3,\e_4 \le B$ and $|\e_7|\ge
  1$. Combining these bounds for the integration over $\e_1, \dots, \e_5,
  \e_7$ with
  \begin{equation*}
    \int_{h(\ee';B) \le 1} \e_1^{-1} \dd \e_6 \dd \e_8 \ll
    \left(\frac{B^{3}}{|\base{1}{1}{0}{2}{6}{0}{5}|}\right)^{1/4}
  \end{equation*}
  by Lemma~\ref{lem:generic_bounds}\eqref{it:generic_bounds_b2} leads to the
  estimation
  \begin{equation*}
    \begin{split}
      V^{(1)}(B) - V^{(2)}(B) \ll{}& \int
      \left(\frac{B^{3}}{|\base{1}{1}{0}{2}{6}{0}{5}|}\right)^{1/4} \dd
      \e_1\cdots\dd \e_5 \dd \e_7\\
      \ll{}& \int \frac{B}{\e_1\e_2\e_3\e_4|\e_7|^{5/4}}\dd \e_1 \cdots
      \dd\e_4\dd\e_7\\
      \ll{}& B(\log B)^4.
    \end{split}
  \end{equation*}

  For $i=2$, we note that $\ee' \in \R_3(B) \setminus \R_2(B)$ implies $|\e_7|
  \le 1$, $0 \le \e_6 \le B/(\base 2232000)$, $\e_5^2 \le \base 3342000/B$ and
  $1 \le \e_1, \dots, \e_4 \le B$. We combine these bounds for the integration
  over $\e_1, \dots, \e_7$ with
  \begin{equation*}
    \int_{h(\ee';B) \le 1} \e_1^{-1} \dd \e_8 \ll \frac{B^{1/2}}{\e_1^{1/2}\e_2^{1/2}|\e_7|^{1/2}}
  \end{equation*}
  by Lemma~\ref{lem:generic_bounds}\eqref{it:generic_bounds_b} for the
  integration over $\e_8$ to obtain
  \begin{equation*}
    \begin{split}
      V^{(4)}(B)-V^{(3)}(B)
      \ll{}&\int \frac{B^{1/2}}{\e_1^{1/2}\e_2^{1/2}} \dd \e_1\cdots \dd \e_6\\
      \ll{}&\int \frac{B^{3/2}}{\base{5/2}{5/2}{3}{2}000} \dd \e_1\cdots \dd
      \e_5\\
      \ll{}&\int \frac{B}{\base 1 1 1 1 0 0 0} \dd \e_1\cdots\dd\e_4\\
      \ll{}&B(\log B)^4.
    \end{split}
  \end{equation*}

  For $i=3$, we note that $\ee' \in \R_4(B) \setminus \R_3(B)$ implies
  $|\e_6|\le 1$, $\e_4^2 \le B/(\base 2 2 3 0 0 0 0)$ and $1 \le \e_1, \e_2,
  \e_3, \e_5 \le B$. We combine these bounds for the integration over $\e_1,
  \dots, \e_6$ with \[\int_{h(\ee';B) \le 1} \e_1^{-1} \dd \e_8 \dd \e_7 \ll
  \frac{B^{2/3}}{\base{1/3}{1/3}0{1/3}1{2/3}0}\] by
  Lemma~\ref{lem:generic_bounds}\eqref{it:generic_bounds_b1} to show that
  \begin{equation*}
    \begin{split}
      V^{(5)}(B)-V^{(4)}(B) &\ll
      \int \frac{B^{2/3}}{\base{1/3}{1/3}0{1/3}100} \dd \e_1 \cdots \dd \e_5\\
      &\ll \int \frac{B}{\base 1 1 1 0 1 0 0} \dd \e_1 \dd \e_2 \dd \e_3 \dd
      \e_5\\
      &\ll B(\log B)^4.
    \end{split}
  \end{equation*}
  This completes the proof.
\end{proof}

Theorem~\ref{thm:qa1a3_main} follows from
Lemma~\ref{lem:qa1a3_summation_rest}, Lemma~\ref{lem:qa1a3_omega_infty} and
Lemma~\ref{lem:qa1a3_manipulation}.

\bibliographystyle{alpha}

\bibliography{bibliography}

\end{document}